\newif\iffinal
\finalfalse % Not final version
\newif\ifarxiv
\arxivtrue  % arxiv version
\documentclass[reqno,twoside,11pt]{amsart}
\usepackage{amsmath}
\usepackage{amsfonts}
\usepackage{amssymb}
\usepackage{verbatim}
\usepackage{epsfig}

\IfFileExists{epsf.def}{\input epsf.def}{\usepackage{epsf}}

\usepackage{color}

\definecolor{Red}{rgb}{1,0,0}
\definecolor{Blue}{rgb}{0,0,1}
\definecolor{Olive}{rgb}{0.41,0.55,0.13}
\definecolor{Yarok}{rgb}{0,0.5,0}
\definecolor{Green}{rgb}{0,1,0}
\definecolor{MGreen}{rgb}{0,0.8,0}
\definecolor{DGreen}{rgb}{0,0.55,0}
\definecolor{Yellow}{rgb}{1,1,0}
\definecolor{Cyan}{rgb}{0,1,1}
\definecolor{Magenta}{rgb}{1,0,1}
\definecolor{Orange}{rgb}{1,.5,0}
\definecolor{Violet}{rgb}{.5,0,.5}
\definecolor{Purple}{rgb}{.75,0,.25}
\definecolor{Brown}{rgb}{.75,.5,.25}
\definecolor{Grey}{rgb}{.5,.5,.5}

\def\red{\color{Red}}

\newcommand{\dist}{\operatorname{dist}}
\newcommand{\dens}{\operatorname{dens}}
\newcommand{\supp}{\operatorname{supp}}

\newcommand{\esssup}{\operatorname{esssup}}

\newcommand{\var}{{\texttt{var}}}

\newcommand{\MM}{\mathcal M}

\newcommand{\N}{\mathbb N}

\newcommand{\BbbP}{\mathbb P}
\newcommand{\BbbE}{\mathbb E}

\newcommand{\R}{\mathbb R}

\newcommand{\Z}{\mathbb Z}
\newcommand{\tH}{\tilde H}

\newcommand{\connom}{\includegraphics[width=0.1in]{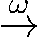}}
\newcommand{\connomp}{\includegraphics[width=0.1in]{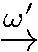}}

\newcommand{\stat}{H}

\newcommand{\ignore}[1]{}

\newtheoremstyle{thm}{1.5ex}{1.5ex}{\itshape\rmfamily}{} {\bfseries\rmfamily}{}{2ex}{}

\newtheoremstyle{def}{1.5ex}{1.5ex}{\slshape\rmfamily}{} {\bfseries\rmfamily}{}{2ex}{}

\newtheoremstyle{rem}{1.3ex}{1.3ex}{\rmfamily}{} {\itshape}
{} {1.5ex}{}

\theoremstyle{thm}
\newtheorem{theorem}{Theorem}[section]
\newtheorem{lemma}[theorem]{Lemma}
\newtheorem{claim}[theorem]{Claim}
\newtheorem{fact}[theorem]{Fact}
\newtheorem{proposition}[theorem]{Proposition}

\newtheorem*{Main Theorem}{Main Theorem.}
\newtheorem{corollary}[theorem]{Corollary}
\newtheorem*{special theorem}{Lindeberg-Feller Theorem for Martingales}

\newtheorem{example}[theorem]{Example}

\theoremstyle{def}
\newtheorem{definition}{Definition}
\newtheorem{assumption}{Assumption}

\theoremstyle{rem}
\newtheorem{remark}{{\itshape Remark}}[]
\newtheorem{iremark}{{\bf Remark}}[]

\numberwithin{equation}{section}

\renewcommand{\section}{\secdef\sct\sect}
\newcommand{\sct}[2][default]{%
\refstepcounter{section} \addcontentsline{toc}{section}{{\tocsection
{}{\thesection}{\!\!\!\!#1\dotfill}}{}} \vspace{0.7cm}
\centerline{\scshape\thesection.\ #1} \nopagebreak \vspace{0.2cm}}
\newcommand{\sect}[1]{%
\vspace{0.4cm} \centerline{\large\scshape\rmfamily #1}
\vspace{0.2cm}}

\renewcommand{\ignore}[1]{{}}

\def\myffrac#1#2 in #3{\raise 2.6pt\hbox{$#3 #1$}\mkern-1.5mu\raise 0.8pt\hbox{$#3/$}\mkern-1.1mu\lower 1.5pt\hbox{$#3 #2$}}

\title[Balanced RWRE]
{\fontsize{14}{20}\selectfont A quenched invariance principle for non-elliptic random walk in i.i.d. balanced random environment}
\author[N.~Berger and J.-D.~Deuschel]
{Noam Berger (Hebrew university of Jerusalem) \\ and Jean-Dominique Deuschel (TU Berlin)}

\begin{document}
\thanks{\hglue-4.5mm\fontsize{9.6}{9.6}\selectfont 
The research was partially supported by grant 2006477 of the Israel-U.S. binational science foundation, by grant 152/2007 of the German Israeli foundation and by ERC StG grant 239990.}

\begin{abstract}
%We prove a quenched invariance principle for random walk in i.i.d. balanced random %environment, without any assumption of ellipticity. The proof uses percolation ideas. 
We consider a random walk on $\Bbb Z^d$ in an i.i.d. balanced random environment, that is
a random walk for which the probability to jump from $x\in\Bbb Z^d$ to nearest neighbor $x+e$ is the same as
to nearest neighbor $x-e$. Assuming that the environment is genuinely $d$-dimensional and balanced
we show a quenched invariance principle: for $P$ almost every environment, the diffusively rescaled
random walk  converges to a Brownian motion with deterministic non-degenerate diffusion matrix.
Within the i.i.d. setting, our result extend both Lawler's  uniformly elliptic result \cite{Lawler82}
and Guo and Zeitouni's elliptic result \cite{GZ} to the general (non elliptic) case. 
\ignore{
In the general ergodic case we provide a mixing example in which no invariance principle (not even a degenerate one) holds.}
Our proof is based on analytic methods and percolation arguments.
% of the Burton-Keane type.
\end{abstract}

\maketitle

\newcommand{\joint}{\Large \texttt{P}}
\newcommand{\quenchedP}{P}
\newcommand{\annealedP}{\BbbP}
\newcommand{\quenchedE}{E}
\newcommand{\annealedE}{\BbbE}
\newcommand{\st}{\mbox{ s.t. }}

\section{Introduction}

This paper deals with a Random Walk in Random Environment (RWRE) on $\Z^d$  which is defined as follows: Let
$\MM^d$ denote the space of all probability measures on the nearest neighbors of the origin $\{\pm e_i\}_{i=1}^d$ and let $\Omega=\left(\MM^d\right)^{\Z^d}$. An {\em environment} is a  point $\omega\in\Omega$, we denote by
$P$ the distribution of the environment on $\Omega$. For the purposes of this paper, we assume that $P$ is an i.i.d. measure, i.e.
\[
P=\nu^{\Z^d}
\]
for some distribution $\nu$ on $\MM^d$.
%and that $P$ is {\em uniformly elliptic}, i.e. there exist $\eta>0$ s.t. for every
%$x\in\Z^d$ and
%$v\in\{\pm e_i\}_{i=1}^d$,
%neighbor $v$ of the origin,
%\[
%P\left(\omega(x,e)<\epsilon\right)=0.
%\]
%\begin{equation}\label{eq:unifelliptic}
%Q(\{\omega:\omega(v)<\eta\})=0.
%\end{equation}
For a given environment $\omega\in\Omega$, the {\em Random Walk} on $\omega$ is a time-homogenous
Markov chain jumping to the nearest neighbors with transition kernel
\[
\quenchedP_\omega\left(\left.X_{n+1}=z+e\right|X_n=z\right)=\omega(z,e)\ge 0, \ \ \ \ \ \ 
\sum_e\omega(z,e)=1.
\]
The {\em quenched law} $\quenchedP_\omega^z$ is defined to be the law on $\left(\Z^d\right)^\N$ induced by
the kernel
$\quenchedP_\omega$ and $\quenchedP_\omega^z(X_0=z)=1$.  We let $\joint^z=P\otimes \quenchedP_\omega^z$ be the joint law of the environment and the walk, and the {\em annealed} law is defined to be its marginal
\[
\annealedP^z=\int_{\Omega}P_\omega^zdP(\omega).
\]

%For simplicity, we omit the superscript when the walk starts from zero.

A comprehensive account of the results and the remaining challenges in the understanding of RWRE can be found in Zeitouni's Saint Flour lecture notes \cite{ofernotes}.

We are interested in the long-time asymptotic behavior of the walk. More precisely
considering the continuous rescaled trajectory $X^N\in C\big(\R^+,\R^d\big)$,
$$X^N_t=\frac{1}{\sqrt{N}}X_{[tN]}+\frac{tN-[tN]}{\sqrt{N}}\big(X_{[tN]+1}-X_{[tN]}\big),\quad t\ge 0,$$
we want to know whether  the {\bf quenched invariance principle holds}, that is,  if for $P$ a.a. $\omega$,
the law of $\{X^N_t\}_{t\ge 0}$ under $P^0_\omega$
converges weakly on $C\big(\R^+;\R^d)$ (endowed with the topology of uniform convergence on every compact interval) to a Brownian motion with  deterministic covariance  matrix.

The invariance principle is a well known classical result for the simple random walk (SRW), cf. \cite{donsker}. 

A satisfying understanding of invariance principles exists for the random conductance model,
which is a reversible RWRE, cf. 
%\cite{MFGW, KV, SS, BB, MP, BD}
\cite{MFGW}, \cite{KV}, \cite{SS}, \cite{BB}, \cite{MP}, \cite{BD}
and many others. 

However in general non-reversible random environments
this question is still   widely open.  
Significant progress has been made in the perturbative regime,  cf. \cite{BK}, \cite{BZ}, \cite{SZ}, in the ballistic regime
cf \cite{BS}, \cite{S2000}, \cite{RAS}, \cite{BeZe} and others,
%(BALLISTIC REFERENCES)
and in the Dirichlet regime
cf \cite{sabot} and others.

By looking at the references above, one can see that the problem of proving an invariance principle is much harder when uniform ellipticity (i.e. that the transition probability between nearest neighbors are bounded away from zero) does not hold. Indeed, in the ballistic regime all the results are proven with the assumption of uniform ellipticity, the perturbative regime is by definition uniformly elliptic and in the reversible regime it had been an open challenge to transfer the uniformly elliptic results of \cite{SS} to less elliptic regimes.

\ignore{\red say what's elliptic. Mention that the perturbative is open in 2d}

% (DIRICHLET REFERENCES).
In this paper we will focus on a special class of environments: the  balanced environment.
% We mostly consider non-elliptic balanced environments.
In particular, we solve the challenge of adapting the methods that were developed for the elliptic case in \cite{Lawler82} and \cite{GZ} to non-elliptic cases.

\begin{definition}\label{def:balanced}
An environment $\omega$ is said to be {\em balanced} if for every $z\in\Z^d$ and neighbor $e$ of the origin, $\omega(z,e)=\omega(z,-e)$.
\end{definition}

Of course we want to make sure that the walk really spans $\Z^d$:  
\begin{definition}\label{def:genddim}
An environment $\omega$ is said to be {\em genuinely $d$-dimensional} if for every neighbor $e$ of the origin, there exists $z\in\Z^d$ such that  $\omega(z,e)>0$.
\end{definition}

Throughout this paper we make the following assumption.
\begin{assumption}\label{ass:main}
$P$-almost surely, $\omega$ is balanced and genuinely $d$-dimensional.
\end{assumption}

Note that whenever the distribution is ergodic, the above assumption is equivalent with
$$P\big[\omega(z,e)=\omega(z,-e)\big]=1,\qquad\text{and}\qquad  P\big[\omega(z,e)>0\big]>0$$
for every $z\in\Z^d$ and a neighbor $e$ of the origin.

Note that unlike \cite{GZ} we do not allow holding times in our model. We do this for the sake of simplicity. Holding times in our case could be handled exactly as they are handled in \cite{GZ}.

%Note that the only ellipticity assumption that we make is that the probability of staying put is zero.
Our main result states
\begin{theorem}\label{thm:main} Assume that the environment is i.i.d., balanced and genuinely $d$-dimensional, then 
 the  quenched invariance principle holds  with a deterministic non-degenerate diagonal covariance matrix.
\end{theorem}

The quenched invariance principle has been derived by Lawler in the 1980-s \cite{Lawler82} for balanced uniform elliptic environments, i.e.,  when there exists $\epsilon_0>0$ such that
$$P(\forall_{z\in\Z^d} \forall_{i=1,...,d}, \ \omega(z,e_i)>\epsilon_0)=1.$$
In fact, Lawler proved this result for general ergodic, uniformly elliptic, balanced environments.

Recently Guo and Zeitouni improved this result in \cite{GZ} for i.i.d  elliptic environments, where
$$P(\forall_{z\in\Z^d} \forall_{i=1,...,d}, \ \omega(z,e_i)>0)=1.$$
Note that our genuinely $d$-dimensional  assumption is much weaker than ellipticity, in particular it applies to the following example

\begin{example}\label{exam:Erwin's walk}
Take $P=\nu^{\Z^d}$ as above with
$$\nu\left[\omega(z,e_i)=\omega(z,-e_i)=\frac12, 
\omega(z,e_j)=\omega(z,-e_j)=0, \forall j\ne i\right]=\frac{1}{d}, \,\, i=1,...,d.$$
In this model, the environment chooses at random one of the $\pm e_i$ direction, see Figure \ref{fig:erwex}).
\end{example}

\begin{figure}[h]
\begin{center}
\includegraphics[width=2in]{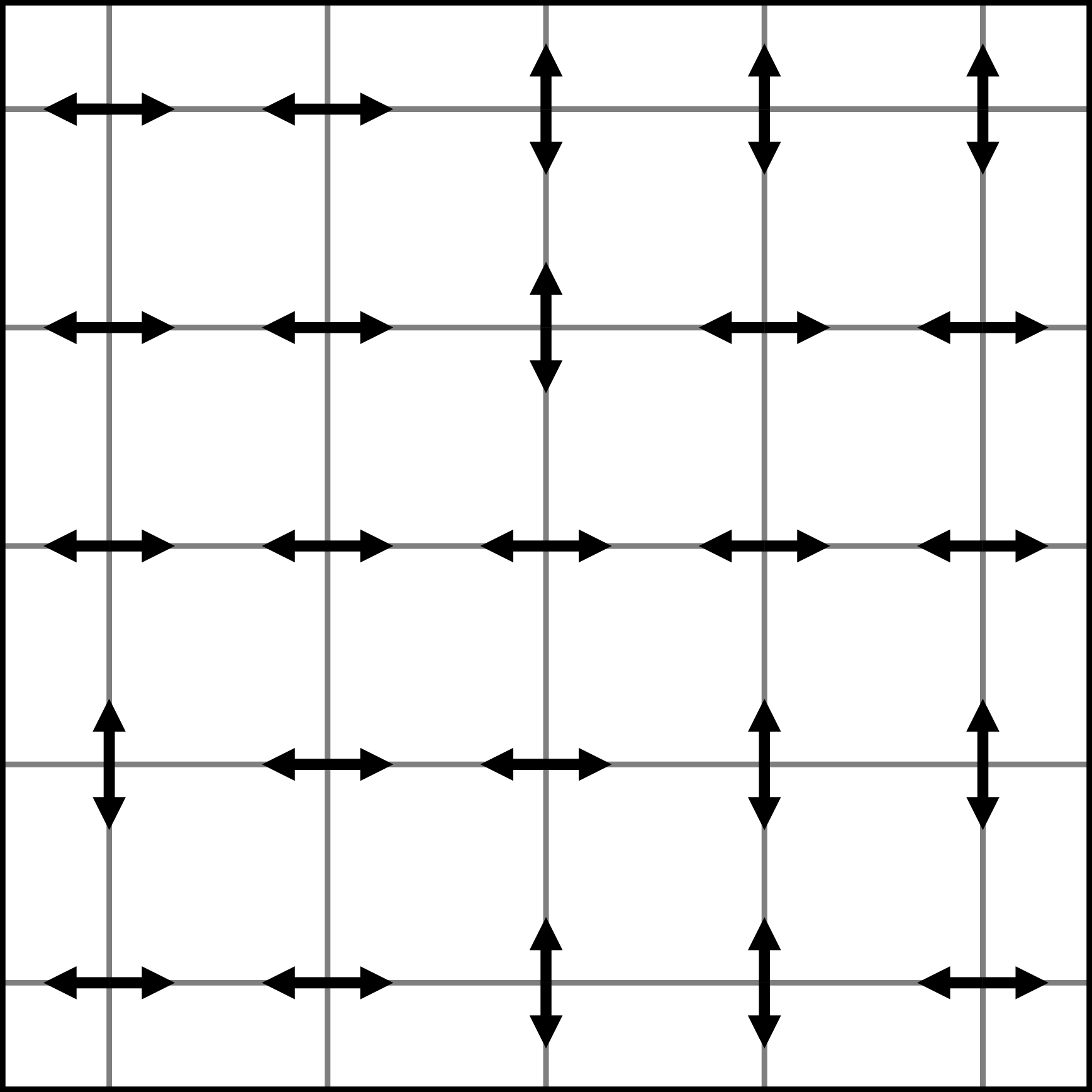}
\caption{\sl An illustration of Example \ref{exam:Erwin's walk} restricted to a small box.
}
\label{fig:erwex}
\end{center}
\end{figure}

\cite{GZ} also shows the quenched invariance principle for ergodic elliptic  environments under the moment condition
$$ E[\big(\prod_{i=1}^d\omega(x,e_i)\big)^{-p/d}]<\infty\quad\text{for some }\quad p>d.$$
Unlike the uniform elliptic case, one can find examples  of ergodic elliptic balanced environment, where the invariance principle fails, as in the following two-dimensional example.

\begin{example}\label{exam:noclt}
For every point $z\in\Z^2$ we have Bernoulli variables $X^{z,ver}_n,
n=1,2,\ldots$ and $X^{z,hor}_n, n=1,2,\ldots$. Those variables are all
independent, and $P(X^{z,ver}_n=1)=P(X^{z,hor}_n=1)=3^{-n}$. Then, for
every $z\in\Z^2$, if $X^{z,ver}_n=1$, then the $2^n$ vertices directly
above $z$ all get chance $1-e^{-2n}$ to move in the vertical direction
and chance $e^{-2n}$ to move in the horizontal direction.
If $X^{z,hor}_n=1$, then the $2^n$ vertices directly to the right of
$z$ all get chance $1-e^{-2n-1}$ to move in the horizontal direction
and chance $e^{-2n-1}$ to move in the vertical direction. If a point
hasn't been spoken for, it gets probability $\frac 14$ to go in each
direction.
If a point has been spoken for more than once, it gets the highest
value assigned to it.
\end{example}
It is not hard to prove that in Example \ref{exam:noclt}, there exist $\alpha$ and $\beta$ positive such that for every large enough $T$, with probability at least $\alpha$, all movements in the time interval $[(1-\beta)T,T]$ are in vertical directions, and with probability at least $\alpha$, all movements in the time interval $[(1-\beta)T,T]$ are in horizontal directions. 
Furthermore, a.s. one can find infinitely many values of $T$ such that all  movements in the time interval $[(1-\beta)T,T]$ are in vertical directions and infinitely many values of $T$ such that all  movements in the time interval $[(1-\beta)T,T]$ are in horizontal directions.
Obviously, such process cannot converge to a Brownian Motion, not even a degenerate one. It is also easy to show that the random walk in Example \ref{exam:noclt} is transient, even though it is two-dimensional.

The balanced assumption is essential for our proof and simplifies the argument greatly. In particular it implies that the walk is a martingale, which enables us to use the vast theory of martingales.

In particular, unlike the case of random conductances, we do not have to define and control a corrector. On the other hand, the existence and properties of an invariant measure for the process w.r.t. the point of view of the particle is a serious difficulty in our case, while it is simple in the case of random conductances.

We now define the process of the point of view of the particle, a notion which is standard in the literature of random walk in random environment, and is used in this paper.
The environment viewed from the point of view of the particle
is the Markov chain $\{\bar\omega_n\}_{n\in\mathbb N}$  given by  
$$\bar\omega_n=\tau_{-X_n}\omega,$$
where $\tau$ is the shift on $\Omega$.

We can also view it as the Markov on $\Omega$ whose generator is

\begin{equation}\label{eq:genpart}
Lf(\omega)=\sum_{e:\|e\|=1}\omega(0,e)\big[f(\tau_{-e}\omega)-f(\omega)\big].
\end{equation}

\ignore{

We now give a short overview of the proof strategy and the main steps of the proof.
%{\red we build on the work of Guo-Zeitouni...} 
The basic idea in our work is the same basic idea as in the work of Lawler \cite{Lawler82} and of Guo and Zeitouni \cite{GZ} -- we use analytical methods to show the existence of a measure which is invariant w.r.t. the random walk and is absolutely continuous w.r.t. to the original measure $P$. However, we have to deal with \ignore{some serious}{several} difficulties that stem from the non-ellipticity. We now give a few more details.

Since the random walk is a martingale, following
the classical argument of Kipnis-Varadhan and Kozlov (cf. \cite{KV}, \cite{Ko}), the quenched invariance principle is established, once we show that 
 the environment viewed from the point of view of the walker has an invariant measure $Q$ {\em which is absolutely continuous} with respect to the original measure $P$ and {\em ergodic}, and such that for $P$-a.e. $\omega$, the random walk from the point of view of the particle arrives a.s. to the support of $Q$.
% the Radon-Nikodym derivative $\frac{dQ}{dP}$. 
The environment viewed from the point of view of the walker
is the Markov chain $\{\bar\omega_n\}_{n\in\mathbb N}$  given by  
$$\bar\omega_n=\tau_{-X_n}\omega,$$
where $\tau$ is the shift on $\Omega$.

We can also view it as the Markov on $\Omega$ whose generator is

\begin{equation}\label{eq:genpart}
Mf(\omega)=\sum_{e:\|e\|=1}\omega(0,e)\big[f(\tau_{-e}\omega)-f(\omega)\big].
\end{equation}

%This is a Markov chain on $\Omega$ under $P$ with transition kernel
%$$M(\omega',d\omega)=\sum_{i=1}^d\big[\omega'(0,e_i)\delta_{\tau_{e_i}\omega'}+
%\omega'(0,-e_i)\delta_{\tau_{-e_i}\omega'}\big].$$
Since $\Omega$ is compact, it is clear that the process $(\bar\omega_n)_{n\in\N}$ has an invariant measure $Q$. The difficulty is to first show that there exists $Q$ which is absolutely continuous with respect to $P$,
and second to check the ergodicity of this $Q$.

In order to verify the absolute continuity, we adapt the argument of \cite{GZ} to our non-elliptic situation. This argument is based on analytical tools first introduced in the
continuous setting by Alexandrov-Bakelman-Pucci and translated into the discrete situation by \cite{Lawler82}, \cite{KT} and \cite{GZ}. A key role is played by a maximal inequality
of the Alexandrov-Bakelman-Pucci type, 
which roughly speaking allows to control the maximum of the some function in terms of the $L^d$ norm of a Laplacian (here the generator of the random walk).
Clearly such an estimate requires some kind of ellipticity, which is lacking for our 
%original
random
walk. Our main tool is to introduce  
a rescaled random walk,  obtained from the original walk  stopped  after each coordinate has been upgraded, and we get the maximum inequality using
the optional stopping theorem for martingales (in this step again the balanced condition is crucial). 
The maximum inequality allows us then  to control for  some $p>1$  the $L^p$- norm of the density of the invariant measure $Q_N$  of the
%(rescaled ?)
walk on the %reflected-periodized
cube of size $N$, uniformly in large  $N$. From this, following the by now standard argument appearing in \cite{Lawler82} and \cite{GZ}, we obtain the existence of an invariant measure $Q\ll P$.

Once existence of $Q\ll P$ is established, we still have to verify the ergodicity of $Q$ for the chain $\{\bar\omega_n\}_{n\in\N}$. Unlike the elliptic case where $Q\ll P$ implies $P\ll Q$,
in our non-elliptic walk we usually have that the support of the measure $Q$, $\text{supp}\ Q$, is strictly smaller than $\text{supp}\ P$  (see figure \ref{fig:noteqviv} on Page \pageref{fig:noteqviv}).
%JDDNEW
\ignore{\red At this point the figure 4 is not very useful since one does not understand why the Q measure is 0 cf. alos 2nd referee, before explaining the sink (below) I would therefore drop the reference to figure 4 here.}
%In the $2$-dimensional case we can apply a simple coupling argument, which unfortunately is not applicable in higher dimensions.

For given $\omega\in \Omega$ let $\supp_\omega Q=\{z\in\Z^d: \tau_{-z}(\omega)\in\text{supp}\  Q\}$. Note that once  $X_m\in \supp_\omega Q$ for some $m\in\N$ then
$X_{n}\in\supp_\omega Q$ for all $n\ge m$.
For a given configuration $\omega\in\Omega$, two points $x,y\in\Z^d$ are  called  strongly
$\omega$-connected if we can find an oriented open path from $x$ to $y$ and an oriented path from $y$ to $x$. 
A subset $A$ of $\Z^d$ is called a {\em sink} if every $x,y$ in $A$ are strongly connected, but there is no oriented path leaving $A$, see Figure \ref{fig:sink}.

\begin{figure}[t]
\begin{center}
\includegraphics[width=4.36in]{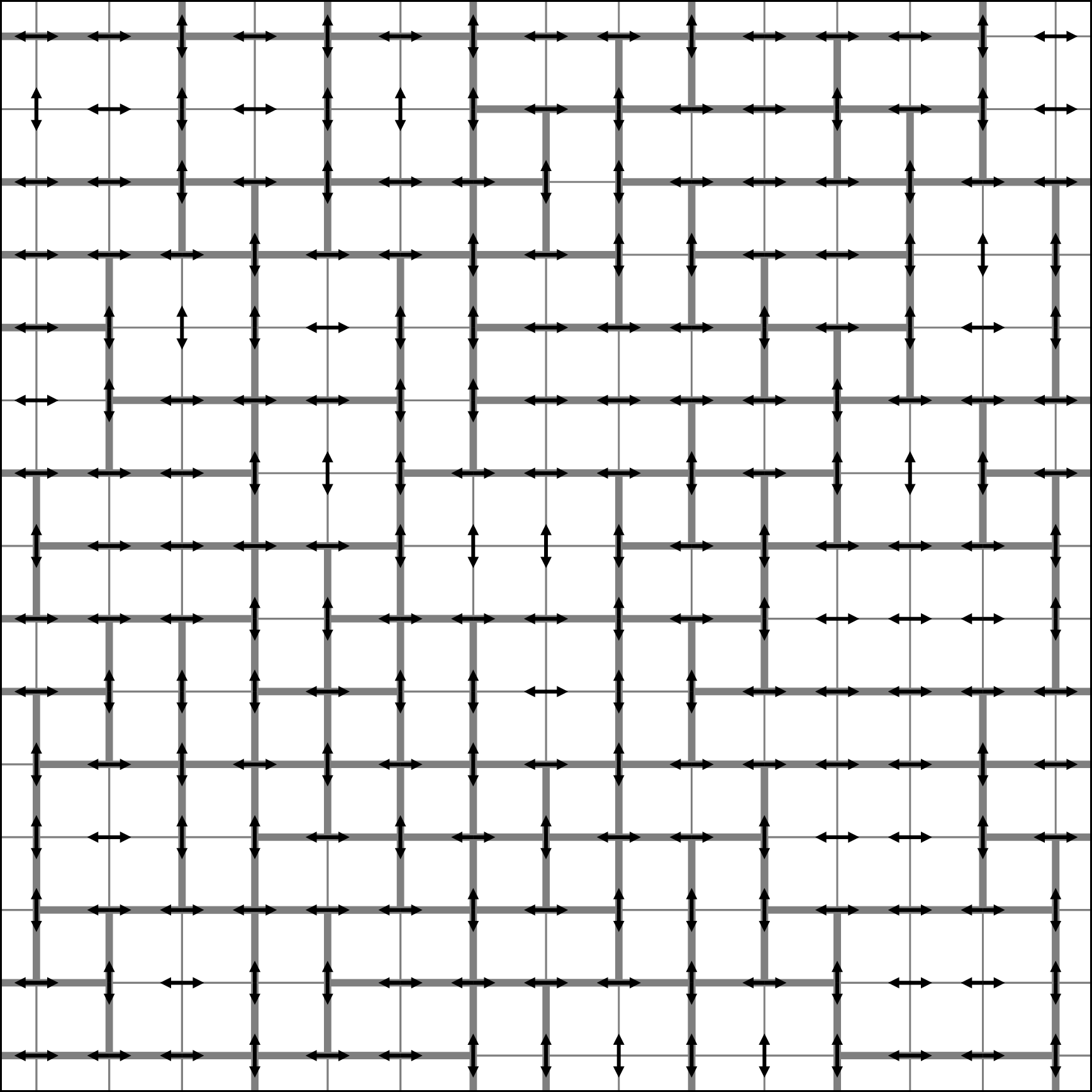}
\caption{\sl 
The gray part in this picture is the sink, and in two dimensions it is quite easy to verify that it is unique and that from every point the walk hits the sink w.p. 1. This result is harder in higher dimensions. It is proved in Section \ref{sec:high_d}.
}
\label{fig:sink}
\end{center}
\end{figure}

It turns out that
the ergodicity of $Q$ is related to the question of the existence and uniqueness of a sink. \ignore{\red This is not true - it is the uniqueness of the sink.}
In the construction of the invariant measure $Q$, we have introduced the reflected-periodized walk in the cube of size $N$. The reason for looking at the reflected-periodic walk instead of the periodic walk as in \cite{Lawler82} and \cite{GZ},
 is to ensure that the support of the invariant measure $Q_N$  and $Q$ are compatible, see Figure \ref{fig:reflect} on Page \pageref{fig:reflect}.
Now our  $L^p$-norm bound on the density of the invariant measure $Q_N$ on the finite cube, yields a lower bound on the size of the support of $Q_N$, and $Q$.
We can then adapt the Burton-Keane argument for uniqueness of a percolation cluster \cite{BuKe} to our setting, even though we only have a weak version of the finite energy condition.
We compensate for the weaker finite energy condition by using density bounds on the support of the invariant measure.

In \cite{HS} percolation problems in the same spirit are dealt with. However, since that paper deals only with the two-dimensional case, we cannot apply their results.
%JDDNEW give a ref of new paper in 2-d ?

Finally we have to make sure that walk starting at $0$ will reach the set  $\supp_\omega Q$ in finite time. 
Here again the argument is based on an  analytical tool: the mean-value inequality, which in the spirit of \cite{GZ} can also be used in the proof of the transience
of the walk in dimensions $d\ge 3$. We do not give the details of the transience proof in this paper.

}

The paper is organized as follows: in section \ref{sec:rescwalk} we introduce the rescaled process and give some estimate on the corresponding stopping times.
Section \ref{sec:max_princ} deals with the maximum principle for the rescaled process, while section \ref{sec:perstat} presents the stationary measure for the periodized environment.
Then in Subsection \ref{sec:stnons} we repeat the arguments from \cite{Lawler82} and \cite{GZ} that lead to the existence of an invariant measure which is absolutely continuous w.r.t. $P$.
%JDDNEW
%Section \ref{sec:coup2d} gives the simple coupling argument in the 2-d case.
 In section \ref{sec:high_d} we 
 \ifarxiv \else
 finally 
 \fi
 prove Theorem \ref{thm:main}.
\ifarxiv
Finally, in Section \ref{sec:mean_value_proof} we we give a proof of the mean value inequality which is 
stated in Section \ref{sec:max_princ} and
used in Section \ref{sec:high_d}.
\else
\fi

 % we deal with the proof of the ergodicity of the invariant measure in higher dimensions, and
%finally in section \ref{sec:mean_value_proof} we give a proof of the mean value inequality which is 
%stated in Section \ref{sec:max_princ} and
%used in Section \ref{sec:high_d}.

\section{The rescaled walk}\label{sec:rescwalk}
In this section we define the rescaled walk, which is a useful notion in the study of non-elliptic balanced RWRE, and prove some basic facts about it.

Let $\{X_n\}_{n=0}^\infty$ be a nearest neighbor walk in $\Z^d$, i.e. a sequence in $\Z^d$ such that $\|X_{n+1}-X_n\|_1=1$ for every $n$. Let $\alpha_n,n\geq 1$ be the coordinate that changes between $X_{n-1}$ and $X_{n}$, i.e. $\alpha(n)=i$ whenever $X_{n}-X_{n-1}=e_i$ or $X_{n}-X_{n-1}=-e_i$.

\begin{definition}\label{def:scrw}
The stopping times $T_k,k\geq 0$ are defined as follows: $T_0=0$. Then
\[
T_{k+1}=\min\left\{
t>T_k:\{\alpha(T_k+1),\ldots,\alpha(t)\}=\{1,\ldots,d\}
\right\}\leq\infty.
\]
We then define the {\em rescaled random walk} to be the sequence (no longer a nearest neighbor walk)
%JDDNEWREV N to n
$Y_n=X_{T_n}$. $\{Y_n\}$ is defined as long as $T_n$ is finite.
\end{definition}

\begin{lemma}\label{lem:well_defined}
$\annealedP$-almost surely, $T_k<\infty$ for every $k$.
\end{lemma}

%\begin{proof}
%Trivial, we will write this at some point.
%\end{proof}

\begin{lemma}\label{lem:tail_est}
There exists a constant $C$ such that for every $n$,
\[
\annealedP(T_1>n)<e^{-Cn^{\frac 13}}.
\]
\end{lemma}
Note that due to lack of stationarity, Lemma \ref{lem:tail_est} does not directly say anything about $T_{k+1}-T_k$ for large values of $k$. In Section \ref{sec:perstat} we will establish estimates for $T_{k+1}-T_k$ for large values of $k$.
\begin{proof}[proof of Lemma \ref{lem:tail_est}]
Note that $W_n=\sum_{i=1}^dX_n^{(i)}$ is a simple random walk, and whenever $W_n$ reaches a new value, $X_n$ visits a new point.
Since the environment is i.i.d., whenever the walk is at a new point, its (annealed) probability of going in any direction, conditioned on its past, is bounded away from zero.
Therefore,
\[
\annealedP\left( T_1>n \left|
\max_{k\leq n}|W_k|\geq n^{\frac 13}
\right.\right)
\leq e^{-Cn^{\frac 13}},
\]

and from standard SRW estimates,
\[
\annealedP\left(
\max_{k\leq n}|W_k|<n^{\frac 13}
\right)
\leq e^{-Cn^{\frac 13}}.
\]
Combined, we get the desired result.
\end{proof}

\begin{proof}[Proof of Lemma \ref{lem:well_defined}]
Assume that almost surely $T_k$ is finite, and we show that almost surely $T_{k+1}<\infty$. By the same argument as in the proof of Lemma \ref{lem:tail_est}, almost surely after time $T_k$ the walk $\{X_n\}$ will visit infinitely many new points. For every coordinate $i$, each time the walk visits a new point, conditioned on the past it has an annealed probability bounded away from zero to make a step in the direction $e_i$. Since infinitely many new points are visited, $\annealedP(T_{k+1}<\infty|T_k<\infty)=1$.
\end{proof}

The annealed estimate in Lemma \ref{lem:tail_est} can easily be turned into a quenched one.
\begin{lemma}\label{lem:quen_tail_est}
\[
P\left(
\omega:\quenchedE_\omega(T_1)>k
\right)
\leq e^{-Ck^{\frac 13}}.
\]
\end{lemma}
\begin{proof}
Note that if $\quenchedE_\omega(T_1)>k$, then
\[
A(\omega)=\sum_{j=k/2}^\infty \quenchedP_\omega(T_1>j) > k/2.
\]

Now, 
\[
E(A(\omega))
=\sum_{j=k/2}^\infty \annealedP(T_1>j)
\leq \sum_{j=k/2}^\infty e^{-Cj^\frac 13}\leq Ck^3 e^{-Ck^\frac 13}
\]
Markov's inequality completes the proof.

\end{proof}

An immediate yet useful corollary of Lemma \ref{lem:quen_tail_est} is the following.
\begin{lemma}\label{lem:for7}
For every $0<p<\infty$, 
\[
E\left[
\quenchedE_\omega(T_1^p)
%JDDNEW
\right]<\infty.
\]
\end{lemma}

\section{A maximum principle and a mean value inequality}\label{sec:max_princ}
In this section we prove a maximum principle which we will later use. It uses the same basic idea as the maximum principle of Kuo and Trudinger, \cite{KT}, \ignore{\red may be mention Alexandrov-Bakelman-Pucci} but the probabilistic and non-elliptic setting requires a new way of estimating the size of the set of the supporting hyperplanes, cf Lemma \ref{lem:ubndi}. We also state a mean value theorem, very similar to Theorem 12 of Guo and Zeitouni \cite{GZ}. The proof of the mean value theorem is very similar to that of Theorem 12 of \cite{GZ}.
\ifarxiv
It appears in Section \ref{sec:mean_value_proof}.
\else
It appears in the arxiv version of this paper, but not in the journal version.
\fi

%{\red REMOVING PROOF OF MEAN VALUE THEOREM}

%Let $\kappa=100$.

For $N\in\N$ and $k=k(N)\in (0,N)\cap\Z$, let $T_1^{(N)}=T_1^{(N,k)}=\min\big(T_1, k\big)$.
Let $h:\Z^d\to\R$ be a real valued function, and for every $z\in\Z^d$, let
%JDDNEW  
$L^{(N)}_\omega h(z):=h(z)-\quenchedE_\omega^z[h(X_{T_1^{(N)}})]$.

Let $Q\subseteq \Z^d$ be finite and connected, and let
$\partial^{(k)}Q=\{z\in \Z^d-Q: \exists_{x\in Q}\|z-x\|_\infty < k\}$.

We say that a 
%JDDNEWREV a a 
 point $z\in Q$ is {\em exposed} if there exists $\beta=\beta(z,h)\in\R^d$ such that 
$h(z)-\langle\beta,z\rangle\geq h(x)-\langle\beta,x\rangle$
for every $x\in Q\cup\partial^{(k)}Q$. We let $D_h$ be the set of exposed points.
%JDDNEWREV
Further, we define the {\em angle of vision} $I_h(z)$ as follows:
\begin{equation}\label{eq:defIz}
%\[
I_h(z)=\left\{
\beta\in\R^d: \forall_{x\in Q\cup\partial^{(k)}Q}  h(x)\leq h(z)+\langle\beta,x-z\rangle
\right\}.
%\]
\end{equation}
This is the set of hyperplanes that touch the graph of $h$ at $(z,h(z))$ and are above the graph of  $h$ all over
$Q\cup\partial^{(k)}Q$. A point $z$ is exposed if and only if $I_h(z)$ is not empty.

%{\bf we could introduce the set $I_h(z)$ already  here?}
% Implemented J.-D.'s idea. Noam
%

\begin{theorem}[Maximum Principle]\label{thm:max_princ}
There exists $N_0$ such that for every $N>N_0$ and every $0<k<N$, every balanced environment 
$\omega$ and every $Q$ of diameter $N$,
%NB: Added restriction on $k$.
if for every $z\in Q$
\begin{equation}\label{eq:condformax}
%\quenchedP_\omega^z\big(T_1>(\log N)^\kappa\big)<e^{-(\log N)^2}
%\quenchedE_\omega^z\big(T_1\cdot{\bf 1}_{T_1>k}\big)<e^{-(\log N)^2}
\quenchedP_\omega^z\big(T_1>k\big)<e^{-(\log N)^3}
\end{equation}
%and
%\begin{equation}\label{eq:condformax2}
%\quenchedE_\omega^z(T_1)<\infty
%\end{equation}
%$P$ almost surely, for all $h$ and all large enough $N$ and $k=N$, 
then
%JDDNEW n to z in h(z)
\[
\max_{z\in Q}h(z)-\max_{z\in \partial^{(k)}Q}h(z)
\leq 6 N \left(
\sum_{z\in Q} {\bf 1}_{z\in D_h}|L^{(N)}_\omega h(z)|^d
\right)^{\frac 1d}
\]
\end{theorem}

If $\Delta_N$ is a cube of side length $N$, then a more convenient way of writing the same thing is
\begin{eqnarray}\label{eq:max_princ}
\nonumber
\max_{z\in \Delta_N}h(z)-\max_{z\in \partial^{(k)}\Delta_N}h(z)
&\leq& 6 N^2 \| {\bf 1}_{ D_h}L^{(N)}_\omega h \|_{\Delta_N,d}\\
&\leq& 6 N^2 \left\| \big(L^{(N)}_\omega h\big)^+ \right\|_{\Delta_N,d}
\end{eqnarray}

where, as in \cite{GZ},
\[
\|f \|_{\Delta_N,p} = \left(
\frac 1{|\Delta_N|}\sum_{z\in \Delta_N}|f(z)|^p
\right)^\frac 1p
\]
is the $L^p$ norm with respect to the uniform probability measure on $\Delta_N$.

\vspace{0.2cm}

\begin{remark}\label{rem:log100}
Note that if $\omega$ is sampled according to an i.i.d. environment satisfying Assumption \ref{ass:main}, then by Lemma \ref{lem:quen_tail_est}, \eqref{eq:condformax} is almost surely satisfied for all large enough $N$, $k=(\log N)^{100}$ and any connected $Q$ of diameter $N$ that contains the origin. However, in this paper we also apply Theorem \ref{thm:max_princ} to environments that are not i.i.d, namely to environments that are the periodized versions of i.i.d. environments.
\end{remark}

We now state a mean value theorem, whose proof, which is essentially the same as the proof of Theorem 12 in \cite{GZ}, 
\ifarxiv
is postponed to Section \ref{sec:mean_value_proof}.
\else
appears in the arxiv version of this paper.
\fi
Let $B_N(x)=\{y\in\Bbb Z^d: |x-y|\le N\}$, and $\bar B_N=B_N\cup\partial^{(\log N)^{100}}B_N$. For $u:\bar B_N\to\R$ let $L_\omega u(z)=u(z)-\quenchedE_{\omega}^z(u(X_1)).$
%JDDNEWREV
%{\bf note that $X_1\notin\bar B_N(x)$ since unbounded, thus $u$ should be defined on the whole of $\Bbb Z^d$}
%
% Disagree - $X_1$ is the first step of the original walk, not the rescaled. Noam
%

\begin{theorem}\label{thm:mean_value}
For any $\sigma\in (0,1), 0<p\le d$ and $x_0\in\Z^d$ we can find $N_0=N_0(\sigma,p,d,x_0)$ and
$C=C(\sigma,p,d)$ such that $P$ almost surely
 if $N\ge N_0$ and  $u$ on $\bar B_N(x_0)$  satisfy
$$L_\omega u(x)=0, \quad x\in B_N(x_0)$$
then
%JDDNEW
$$\max_{B_{\sigma N(x_0)}}u \le C\|u^+\|_{B_N(x_0),p}.$$
\end{theorem}

\begin{proof}[Proof of Theorem \ref{thm:max_princ}]
As in \cite{KT}, we are mostly concerned with the angle of vision in any vertex, defined as follows:
Let $z\in Q$. Recall The angle of vision $I_h(z)$ as defined as in \eqref{eq:defIz}.
%\[
%I_h(z)=\left\{
%\beta\in\R^d: \forall_{x\in Q\cup\partial^{(k)}Q}  h(x)\leq h(z)+\langle\beta,x-z\rangle
%\right\}
%\]
%This is the set of hyperplanes that touch the graph of $h$ at $(z,h(z))$ and are above the graph of  $h$ all over
%$Q\cup\partial^{(k)}Q$.

Equivalently to \cite{KT}, we will now state and use two simple geometrical lemmas. The proofs of these lemmas are postponed to immediately after the end of the current proof.

\begin{lemma}\label{lem:lbndi}
For every $N$ and $0<k<N$,
\[
\lambda\left(
\bigcup_{z\in Q}I_h(z)
\right)
\geq \left|\frac{\max_{z\in Q}h(z)-\max_{z\in \partial^{(k)}Q}h(z)}{2N}\right|^d,
\]
where $\lambda$ is Lebesgue's measure in $d$ dimensions.
\end{lemma}

\begin{lemma}\label{lem:ubndi}
Almost surely, for every large enough $N$, every $Q$ of diameter $N$, every $\omega$ satisfying \eqref{eq:condformax} and every $z\in Q\cap D_h$, 
\begin{equation}\label{eq:ubndi}
\lambda(I_h(z))\leq \left[\big(3L^{(N)}_\omega h(z)\big)^+\right]^d.
\end{equation}
\end{lemma}

The theorem now follows once we note that 
\[
\lambda\left(
\bigcup_{z\in \Delta_N}I_h(z)
\right)
\leq
\sum_{z\in \Delta_N}\lambda(I_h(z)).
\]

\end{proof}

\begin{proof}[Proof of Lemma \ref{lem:lbndi}]
This is identical to the proof of Lemma 2.2 in \cite{KT}.
%As in \cite{KT}, we will write it later.
\end{proof}

\begin{proof}[Proof of Lemma \ref{lem:ubndi}]
%We first prove the statement for $T_1$ rather than $T_1^{(N)}$ (with a slightly better constant), and then see how to convert this into the required statement.
Let $\beta\in I_h(z)$. Fix $i\in\{1,\ldots,d\}.$ For a walk $\{X_n\}$ and $i=1,\ldots,d$, let $u_i=\min\{n:\alpha(n)=i\}\leq\infty$. We define the events
\[
A_i^{(+)}=\{X_{u_i}-X_{u_i-1}=e_i\ \mbox{and}\ u_i\leq T_1^{(N)}\}
\]
and
\[
A_i^{(-)}=\{X_{u_i}-X_{u_i-1}=-e_i\ \mbox{and}\ u_i\leq T_1^{(N)}\}. %\}={A_i^{(+)}}^c.
\]

%Note that by \eqref{eq:condformax},
%$\quenchedP_\omega^z(A_i^{(+)})=\quenchedP_\omega^z(A_i^{(-)})=1/2$. Note also that again %by \eqref{eq:condformax}, we have

Let $W$ be a random variable which takes $+1$ with probability $1/2$ and $-1$ with the same probability, and is independent of the walk. Let $A_i^0$ be the event $A_i^0=\{u_i>k\}.$ We define 
\[
A_{i,N}^{(+)}=\big(A_{i}^{(+)}\cap (A_i^0)^c\big)
\cup
\big(\{W=+1\}\cap (A_i^0)\big)
\]
and equivalently
\[
A_{i,N}^{(-)}=\big(A_{i}^{(-)}\cap (A_i^0)^c\big)
\cup
\big(\{W=-1\}\cap (A_i^0)\big).
\]

Note that $\quenchedP_\omega^z\big(A_{i,N}^{(+)}\big)=\quenchedP_\omega^z\big(A_{i,N}^{(-)}\big)=1/2$ and that $A_{i,N}^{(+)}$ and $A_{i,N}^{(-)}$ are disjoint events. Therefore,
\begin{equation}\label{eq:oomega}
\quenchedE_\omega^z(X_{T_1^{(N)}}|A_{i,N}^{(+)}) - z 
=
z - \quenchedE_\omega^z(X_{T_1^{(N)}}|A_{i,N}^{(-)}).
\end{equation}
Let $O^{(i)}_\omega(z)=\quenchedE_\omega^z(X_{T_1^{(N)}}|A_{i,N}^{(+)}) - z$.

\ignore{
Let $\bar u_i=\min(u_i,T_1^{(N)})$. Note that
$\quenchedE_\omega^z(\bar u_i)\leq\quenchedE_\omega^z(T_1)<\infty$.
Therefore, using the optional sampling theorem, we get that
for $P$-almost every $\omega$ and every point $z$,
\begin{equation}\label{eq:a+ave}
\quenchedE_\omega^z(X_{T_1}|A_i^{(+)})=z+e_i
\end{equation}

and

\begin{equation}\label{eq:a-ave}
\quenchedE_\omega^z(X_{T_1}|A_i^{(-)})=z-e_i.
\end{equation}
}

$\beta\in I_h(z)$, and therefore $\langle \beta,x\rangle \geq h(z+x)-h(z)$  for every $x\in Q\cup\partial^{(k)}Q$.
%JDDNEWREV which $x$?
%
%  Specified which values of $x$. Noam
%
In particular, using the definition of $O^{(i)}_\omega(z)$ and \eqref{eq:oomega},

%using \eqref{eq:a+ave} and \eqref{eq:a-ave},
\begin{eqnarray}\label{eq:cube1}
\nonumber
&&\langle \beta,O^{(i)}_\omega(z) \rangle = \sum_{x\in Q\cup\partial^{(k)}Q} \langle \beta,x\rangle \quenchedP_\omega^z(X_{T^{(N)}_1}=x+z
|A_i^{(+)})\\
& \geq& \sum_{x\in Q\cup\partial^{(k)}Q} (h(x+z)-h(z)) \quenchedP_\omega^z(X_{T^{(N)}_1}=x
|A_i^{(+)})
\end{eqnarray}
%JDDNEWREV which $x$, why is summation finite and well defined? also $T$ to $T^{(N)}$
%
% The summation is finite and well defined because we only sum over $x$ in $x\in Q\cup\partial^{(k)}Q$, which is
% a finite set.  Noam
%

Equivalently,
\begin{eqnarray}\label{eq:cube2}
\nonumber
&&\langle \beta,-O^{(i)}_\omega(z) \rangle = \sum_x \langle \beta,x\rangle \quenchedP_\omega^z(X_{T^{(N)}_1}=x+z
|A_i^{(-)})\\
& \geq& \sum_{x\in Q\cup\partial^{(k)}Q} (h(x+z)-h(z)) \quenchedP_\omega^z(X_{T^{(N)}_1}=x
|A_i^{(-)})
\end{eqnarray}

in other words,
\begin{eqnarray}\label{eq:cube3}
\nonumber
&&\sum_{x\in Q\cup\partial^{(k)}Q} (h(x+z)-h(z)) \quenchedP_\omega^z(X_{T^{(N)}_1}=x
|A_i^{(-)})\\
\nonumber
&\leq& \langle \beta,O^{(i)}_\omega(z) \rangle\\
 &\leq& -\sum_{x\in Q\cup\partial^{(k)}Q} (h(x+z)-h(z)) \quenchedP_\omega^z(X_{T^{(N)}_1}=x
|A_i^{(+)}),
\end{eqnarray}

so whenever $\beta$ exists, $\langle \beta,O^{(i)}_\omega(z) \rangle$ is in an interval of length bounded by
\begin{eqnarray*}
&-&\left[\sum_x (h(x+z)-h(z)) \quenchedP_\omega^z(X_{T^{(N)}_1}=x
|A_i^{(+)})
 +\sum_x (h(x+z)-h(z)) \quenchedP_\omega^z(X_{T_1}=x
|A_i^{(-)})
\right]\\
&=&2\sum_x (h(z)-h(x+z)) \quenchedP_\omega^z(X_{T_1}=x)
=2L^{(N)}_\omega h(z),
\end{eqnarray*}
where the summation is over $x\in Q\cup\partial^{(k)}Q$.
In particular, $L^{(N)}_\omega h(z)$ is non-negative if $\beta$ exists.

Therefore, $\lambda(I_h(z))$ is bounded by the volume of the parallelogram
\[
L=\left\{\gamma\in\R^d: \forall_i \ 0\leq \big\langle\gamma,O^{(i)}_\omega(z)\big\rangle\leq 2L^{(N)}_\omega h(z) \right\}.
\]
We thus need to estimate the volume of the parallelogram $L$. By standard linear algebra,
\[
\lambda(L) = (2L^{(N)}_\omega h(z))^d\det(M^{-1})
\]
where $M$ is the matrix whose columns are the vectors $O^{(i)}_\omega(z),\ 0\leq i\leq d$.
Therefore, we need to estimate the value of the vectors $O^{(i)}_\omega(z)$.
%We state a claim regarding these values, which we prove after we finish the proof of the lemma.
\begin{claim}\label{claim:oomega}
for every $i$, 
\[
\|e_i-O^{(i)}_\omega(z)\| < e^{-(\log N)^2}.
\]
\end{claim}

Noting that the determinant is a continuous function, we get that \eqref{eq:ubndi} holds for all large enough $N$.

\end{proof}

\begin{proof}[Proof of Claim \ref{claim:oomega}]
We calculate separately $O_i=\langle O^{(i)}_\omega(z), e_i\rangle$ and $O_{\not i}=O^{(i)}_\omega(z) - O_i$.
 
By the optional sampling theorem,
\begin{eqnarray*}
O_i &=& \quenchedP_\omega^z(A_i^{(+)}|A_{i,N}^{(+)})
\left[\quenchedE_\omega^z(X_{T_1^{(N)}}|A_{i}^{(+)}) - z\right]\\
&+&\quenchedP_\omega^z(A_i^{0}|A_{i,N}^{(+)})
\left[\quenchedE_\omega^z(X_{T_1^{(N)}}|A_{i,N}^{0}) - z\right]\\
&= & \quenchedP_\omega^z(A_i^{(+)}|A_{i,N}^{(+)}),
\end{eqnarray*}
and therefore 
\begin{equation}\label{eq:oi}
|O_i-1|<e^{-(\log N)^3}
\end{equation}

Using the optional sampling theorem one more time,

\begin{eqnarray}\label{eq:onoti}
O_{\not i} = 0.
\end{eqnarray}

The claim follows from \eqref{eq:oi} and \eqref{eq:onoti}.

\end{proof}

\begin{remark}\label{rem:balancedmart}
Note that the rescaled walk is balanced in the following sense: For every $x$, $\omega$, $N$ and $k$,
\[
\sum_{y}(y-x)\quenchedP_\omega^x\left(X_{T_1^{(N)}}=y\right)=0.
\]
\end{remark}

\section{Stationary measure for the periodized environment}\label{sec:perstat}
As in \cite{GZ} and \cite{Lawler82}, in this section we analyze the stationary measure of the walk on a periodized environment. Unlike those papers, here we consider a slight variation of the periodized environment, namely the reflected periodized environment, see Figure \ref{fig:reflect}. %The advantage of this choice will only be apparent after Section \ref{sec:coup2d}.
The advantage of the choice of the reflected periodized environment over the one appearing in 
\cite{GZ} and \cite{Lawler82} is that every walk in the reflected periodized environment is (up to, possibly,  some holding times) a legal walk in the original environment, which is not the case for the periodized environment appearing in \cite{GZ} and \cite{Lawler82}. This property of the reflected periodized environment will turn out to be very useful in Section \ref{sec:high_d}.

\begin{figure}[t]
\begin{center}
\includegraphics[width=4.36in]{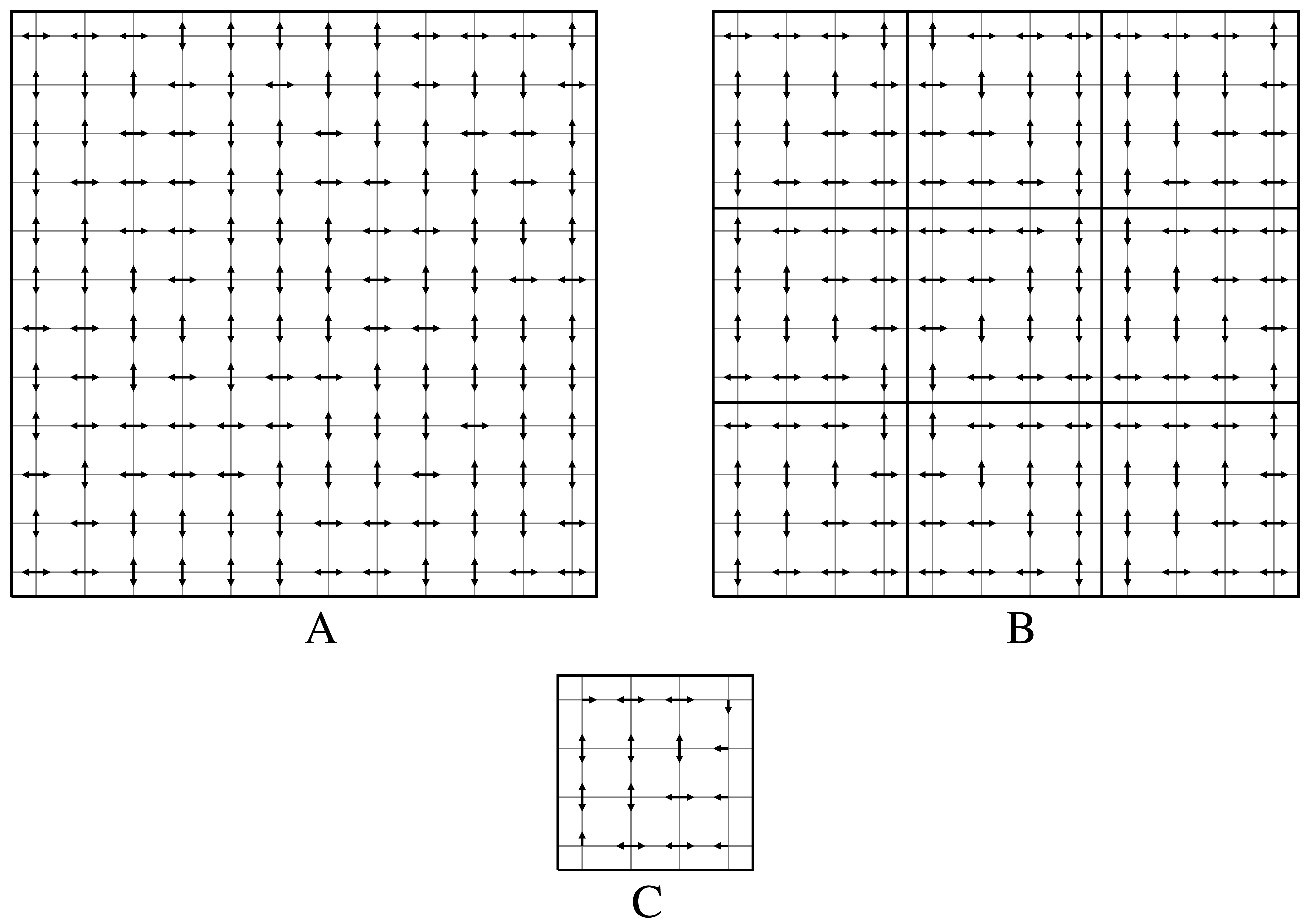}
\caption{\sl 
The configuration under the letter {\em A} is the original configuration. The configuration under the letter {\em B} is the (reflected) periodized configuration with period 4, and the configuration under the letter {\em C} is the effective environment for the reflected random walk in the $4\times 4$ box. In places where there is only an arrow pointing in one direction, the walker stays put with probability $\frac 12.$ The origin in this picture is at the upper left corner.
}
\label{fig:reflect}
\end{center}
\end{figure}

The conclusion of this section is that for some $p>1$, the $L^p$ norm of the Radon-Nikodym derivative of a stationary measure with respect to the uniform measure on the period-cube is bounded as a function of the size of the period. As in \cite{GZ} and \cite{Lawler82} this will turn out to be the crucial step in the way of proving a CLT.

Differently from \cite{GZ} and \cite{Lawler82}, we do it here with the stationary measure w.r.t. the rescaled walk and not w.r.t. the original walk, because the original walk does not necessarily obey the maximum principle (Theorem \ref{thm:max_princ}). The main idea is an idea that we learned from Theorem 5 of \cite{GZ}, but as we work with the rescaled walk, which is less regular than the original walk, the whole argument becomes significantly more complex. In Subsection \ref{sec:non-rescaled} we transfer the result from stationary measures w.r.t. the rescaled walk to stationary measures w.r.t. the original walk.
%we will need some preparation before we get to this main step, and the main step itself will have to be coupled with some sort of a bootstrap argument.

\subsection{Definition of the periodized environment}
\newcommand{\frN}{2N}
For every environment $\omega\in\Omega$ and $N\in\N$, we define the periodized environment 
$\omega^{(N)}$ as follows:

First we define $\omega^{(N)}(z)$ for $z$ in the cube $[0,2N-1]^d$: for $z=(z_1,z_2,\ldots,z_d)$
we define $\omega^{(N)}(z)=\omega(z^\prime)$ where 
\[
z^\prime = \big(
\min(z_1, 2N-1-z_1),\min(z_2, 2N-1-z_2),\ldots,\min(z_d, 2N-1-z_d)
\big).
\]

Then for general $z$ we define $\omega^{(N)}(z)=\omega^{(N)}(z \mod 2N)$ where for every coordinate $i$ we define
$(z\mod 2N)_i:=z_i-2N\cdot\lfloor\frac{z_i}{2N}\rfloor$.  
%JDDNEW the the
For a given environment $\omega$ and $N\in\N$, let $P_{\omega,N}$ be the  uniform distribution over all $(2N)^d$ shifts of $\omega^{(N)}$. By $E_{\omega,N}$ we denote the expectation with respect to the distribution $P_{\omega,N}$.
%Sometimes, with slight abuse of notation, we use $P_{\omega,N}$ and $E_{\omega,N}$ as the distribution and the expectation of the random walk on an environment which is sampled by $P_{\omega,N}$.
 As in \cite{GZ}, due to the ergodic theorem and to the fact that the planes of reflection are a negligible set, $P$-almost surely $P_{\omega,N}$ converges weakly to $P$.
 
Note that the random walk in $\Z^d$ under $\omega^{(N)}$ corresponds to the reflected random walk on $\Delta_N=[0,N)^d$ under $\omega$, with some holding times. Indeed, if we define the function $f:\Z^d\to \Delta_N$ to be
\begin{eqnarray}\label{eq:deff}
\nonumber
f(z)&=&(g(z_1),\ldots,g(z_d)), \mbox{ where}\\
  g(x):&=&\min\big( x\mod 2N, 2N-1-(x\mod 2N) \big),
\end{eqnarray}
then $\{f(X_n)\}_{n=1}^\infty$ follows the law of a random walk on $\Delta_N$ under $\omega$ which is reflected at the boundaries of the cube, with a holding time when the random walker wants to leave the cube (again, see Figure \ref{fig:reflect}).

Therefore, we can state and prove lemmas similar to Lemmas \ref{lem:tail_est} and \ref{lem:quen_tail_est}.

\begin{lemma}\label{lem:ntail_est}
There exists a constant $C$ such that for every $z$, every $N$ and every $k<\frN$,
\[
\int_{\Omega}\quenchedP^z_{\omega^{(N)}}(T_1>k)dP(\omega)
<e^{-Ck^{\frac 13}}.
\]
\end{lemma}

%{\red LEMMA \ref{lem:quen_ntail_est} CHANGED -- INSTEAD OF FIRST MOMENT WE CONTROL THE SECOND MOMENT. %J.-D., PLEASE VERIFY CORRECTNESS.}

\begin{lemma}\label{lem:quen_ntail_est} There exists a constant $C$ such that for every $z$, every $N$ and every $k<\frN$,
\[
P\left(
\omega:\quenchedE^z_{\omega^{(N)}}\left[(T_1\wedge \frac{N}{2})^2\right]>k
\right)
\leq e^{-Ck^{\frac 16}}.
\]
\end{lemma}
%JDDNEWREV  Why $2N$ to $N/2$
\begin{proof}[Proof of Lemma \ref{lem:ntail_est}]
The proof of Lemma \ref{lem:ntail_est} is basically the same as that of Lemma \ref{lem:tail_est}, 
except that we need to handle the fact that the environment is not i.i.d. and not even locally i.i.d. (consider, for example, any neighborhood of the point 0). As in the proof of Lemma \ref{lem:tail_est}, let $W_n=\sum_{i=1}^dX_n^{(i)}$. It is enough to show that, for two appropriate constants $C_1$ and $C_2$, the probability that  the reflected walk $f(X_k)$ (see display \eqref{eq:deff} for the definition) visits less than $C_1k^{1/3}$ points up to time $k$ is bounded by $e^{-C_2k^{1/3}}$.

To this end we consider separately the coordinates for which the point $z$ is closer than $\frac{k^{1/3}}d$ to the boundary of $\Delta_N$ and those for which the point $z$ is further than $\frac{k^{1/3}}d$ from the boundary. Without loss of generality, assume that $0\leq z^{(i)}<\frac{k^{1/3}}d$ for $1\leq i\leq \ell$, and that $\frac{k^{1/3}}d\leq z^{(i)}\leq N-\frac{k^{1/3}}d$ for $\ell<i\leq d$.

let 
$Z_n = W_n-W_0$ be the change in $(W_\cdot)$.
With probability greater than $1-\exp(-C_2k^{1/3})$, we get that
\[
\max_{n\leq k}|Z_n|>3k^{1/3}.
\]
Therefore, there exists a coordinate $i$ such that 
\[
\max_{n\leq k}|X_n^{(i)}-z^{(i)}|>\frac{3k^{1/3}}d.
\]
%JDDNEW
Now, if $\ell<i\leq d$, then the first $\frac{k^{1/3}}d$ times that $|X_n^{(i)}-z^{(i)}|$ reaches a new maximum, $f(X_n)$ visits a new point. If $1\leq i\leq \ell$, then whenever $|X_n^{(i)}-z^{(i)}|$ reaches $\frac{2k^{1/3}}d+1,\frac{2k^{1/3}}d+2,\ldots,\frac{3k^{1/3}}d$, the process $f(X_n)$ visits a new point.

\end{proof}

Lemma \ref{lem:quen_ntail_est} follows from Lemma \ref{lem:ntail_est} using the exact same calculation that yields Lemma \ref{lem:quen_tail_est} from Lemma \ref{lem:tail_est}. The different power ($1/6$ instead of $1/3$) stems from the power $2$ inside the expectation.

\subsection{Empirical distribution of $\quenchedE_\omega^z(T_1\wedge\frac N2)$}
For a number $N$ and an environment $\omega$, we denote $T=T(\omega,N):=\sqrt{E^0_\omega\big(\min(T_1,\frac N2)^2\big)}$
and for $z\in\Z^d$ we denote $T^z=T^z(\omega,N):=\sqrt{E^z_\omega\big(\min(T_1,\frac N2)^2\big)}$.

\begin{lemma}\label{lem:emp_dist} Fix $1\leq p<\infty$.
$P$-almost surely, for all $N$ large enough, for all $k\leq(\log\log N)^{100}$,
\begin{equation}\label{eq:emp_dist}
E_{\omega,N}(T^p\ ;\ T>k) \leq e^{-Ck^{1/3}}
\end{equation}
where $E(X\ ;\ A)$ is defined to be $E(X\cdot{\bf 1}_A)$.
\end{lemma}

\begin{proof} % {\Large NEED TO WORK ON THIS PROOF}
%By Lemma \ref{lem:tail_est} and the ergodic theorem,
%We use the following claim, whose proof is postponed to the end of the current proof.
%\begin{claim}\label{claim:pmoment}
%From Lemma \ref{lem:quen_ntail_est} and some standard mixing arguments, we learn that there exists $c$ such that $P$-%almost surely for all $N$ large enough,

First, we show that there exists $c$ such that $P$-almost surely for all $N$ large enough,

\begin{equation}\label{eq:fixed_part}
E_{\omega,N}(T^{2p}) <c.
\end{equation}

%From Lemma \ref{lem:ntail_est},
%\[
%E\left(\right)
%\]
Indeed, the LHS of \eqref{eq:fixed_part} equals
%JDDNEWREV n to $2N$
\begin{equation}\label{eq:kfp}
\frac{1}{|\Delta_{2N}|}\sum_{z\in\Delta_{2N}}\left[T^z(\omega^{(N)},N)\right]^{2p}
=\frac{1}{|\Delta_{2N}|}\sum_{z\in\Delta_{2N}}\left[\quenchedE_{\omega^{(N)}}^z\big(\min(T_1,\frac N2)^2\big)\right]^{p}
\end{equation}
%JDDNEWREV you mean $\Delta_N$ or $\Delta_{2N}?$
%
% Changed indeed. Noam
%
%
Let $D_N=\{z\in\Delta_{2N}:\dist(z,\partial\Delta_{2N}>N^{0.7})\}$. Then for all $z\in D_N$, the probability that the random walk starting at $z$ reaches the boundary of $\Delta_{2N}$ before time $\frac N2$ decays like $e^{-cN^{0.4}}$.
%JDDNEWREV really, why? at least like?
%
% It is essentially because the RW on the diagonal is a SRW, and for SRW this estimate is precise.  Noam
%
%
 Therefore, for every $z\in D_N$, we get that 
\[
\quenchedE_{\omega^{(N)}}^z\left[\min(T_1,N/2)^2\right]
\leq \quenchedE_{\omega}^z\left[\min(T_1,N/2)^2\right] + N^2e^{-N^{0.4}}
\leq \quenchedE_{\omega}^zT_1^2 + N^2e^{-N^{0.4}}
\]

and by applying Lemma \ref{lem:for7} and the ergodic Theorem to the i.i.d. environment $\omega$ we get that a.s.
\begin{equation}\label{eq:inside}
%JDDNEWREV n to 2N
\sup\left\{
\frac{1}{|\Delta_{2N}|}\sum_{z\in D_{N}}\left[\quenchedE_{\omega^{(N)}}^z\big(\min(T_1,\frac N2)^2\big)\right]^{p} \ : \ 
N\in\N\right\} < \infty.
\end{equation}

We thus need to bound
\[
\frac{1}{|\Delta_{2N}|}\sum_{z\in\Delta_{2N}\setminus D_N}\left[\quenchedE_{\omega^{(N)}}^z\big(\min(T_1,\frac N2)^2\big)\right]^{p}.
\]
To this end, we use the fact that 
$
|\Delta_{2N}\setminus D_N| / |\Delta_{2N}|<CN^{-0.3},
$
Lemma \ref{lem:quen_ntail_est} with choice of parameter $k=(\log N)^{20}$ 
%JDDNEWREV what hapens to smaller k?
and Borel-Cantelli.

%\end{claim}

%Indeed, to see \eqref{eq:fixed_part}, 

%Indeed, the random walk in $Z^d$ under the environment $\omega^{(N)}$ is equivalent to the %random walk in the cube $\Delta_N$ with reflecting boundary conditions under the environment $%\omega$. 

Now that \eqref{eq:fixed_part} has been established, by Cauchy-Schwarz, all we need to show is that 
$P$-almost surely, for all $N$ large enough, for all $k\leq(\log\log N)^{100}$,

\begin{equation}\label{eq:k_part}
P_{\omega,N}(T>k) \leq e^{-Ck^{1/3}}.
\end{equation}

Note that 
\[
P_{\omega,N}(T>k) = \frac{1}{|\Delta_{2N}|}\sum_{z\in \Delta_{2N}}{\bf 1}_{\{T^z({\omega^{(N)}})>k\}}.
\]

To prove \eqref{eq:k_part}, we need a second moment estimate.
Let $\ell$ be an integer number, whose value will be determined later. Then
\begin{equation}\label{eq:decoup}
T^{z}(\omega)=\sum_{h=1}^{N/2} h\quenchedP_\omega^z(T_1=h)
=\sum_{h=1}^{\ell-1} h\quenchedP_\omega^z(T_1=h) + \sum_{h=\ell}^{N/2} h\quenchedP_\omega^z(T_1=h)
\end{equation}

Write 
\[
B^{\ell,N/2}_\omega(z)=\sum_{h=\ell}^{N/2} h\quenchedP_\omega^z(T_1=h)
\]
and
\[
B^{1,\ell}_\omega(z)=\sum_{h=1}^{\ell-1} h\quenchedP_\omega^z(T_1=h).
\]

By Lemma \ref{lem:ntail_est},
\begin{eqnarray}\label{eq:bigell}
\nonumber
E\big(B^{\ell,N/2}_\omega(z)\big)
&=&E\left(\sum_{h=\ell}^{N/2} h\quenchedP_\omega^z(T_1=h)\right)\\
&=&\sum_{h=\ell}^{N/2} h\annealedP^z(T_1=h)\leq C\ell^3e^{-\ell^{1/3}}.
%&=&C(\log N)^{180}e^{-(\log N)^{\frac {60}3}}.
\end{eqnarray}

Now set $\ell=\big[ (\log N)^{60} \big]$.
Using Markov's inequality and a union bound, from \eqref{eq:bigell} we see that
\begin{eqnarray*}
P\left(
\exists_{z\in \Delta_{2N}}B^{\ell,{N/2}}_\omega(z)>1
\right)
C\leq |\Delta_{2N}|\ell^3e^{-\ell^{1/3}}\leq C e^{-(\log N)^{18}}
\end{eqnarray*}
and by Borel-Cantelli, with probability 1, $B^{\ell,N/2}_\omega(z)\leq 1$ for all $N$ large enough and every $z\in \Delta_{2N}$.

Therefore, it is sufficient to show that almost surely for all large enough $N$ and all $k\leq(\log\log N)^{100}$,
\begin{equation}\label{eq:total_sum}
T_{N,k}:=\frac{1}{|\Delta_{2N}|}\left[\sum_{z\in \Delta_{2N}}
{\bf 1}_{\{B^{1,\ell}_{\omega^{(N)}}(z)>k-1\}}\right]\leq e^{-Ck^{1/3}}.
%\frac{1}{|\Delta_N|}\sum_{z\in \Delta_N}{\bf 1}_{\{B^{1,\ell}_{\omega^{(N)}}(z)>k-1\}}
\end{equation}

From Lemma \ref{lem:quen_ntail_est}, for every $z\in \Delta_{2N}$,
%JDDNEWREV T to B
\begin{equation}\label{eq:1st_mom}
P(B_{\omega^{(N)}}^{1,\ell}(z)>k-1)\leq
P(T^z(\omega^{(N)})>k-1)\leq e^{-Ck^{1/3}}:=f(k).
\end{equation}

\ignore{
#######################
We now state a claim regarding the dependence of $T^z$ and $T^w$ when $z$ and $w$ are far away from each other.
\begin{claim}\label{claim:decouple}
Whenever $\|z-w\|>(\log N)^{60}$,
\begin{equation}\label{eq:2nd_mom}
P(T^z(\omega^{(N)})>k; T^w(\omega^{(N)})>k)\leq e^{-C(\log N)^2}+P(T^z(\omega^{(N)})>k)^2.
\end{equation}
\end{claim}

The proof of Claim \ref{claim:decouple} is postponed to the end of this subsection.
##########################
}

Clearly, for every $z$ and $w$, 
\begin{eqnarray}\label{eq:zwclose}
\nonumber
&&P\left(B_{\omega^{(N)}}^{1,\ell}(z)>k-1\ ;\ P(B_{\omega^{(N)}}^{1,\ell}(w)>k-1\right)\\
&\leq& P\big(B_{\omega^{(N)}}^{1,\ell}(z)>k-1\big)\leq f(k).
\end{eqnarray}

If, in addition, $||z-w||>\ell$, then by the i.i.d. nature of $P$ we get that
\begin{eqnarray}\label{eq:zwfar}
\nonumber
&&P\left(B_{\omega^{(N)}}^{1,\ell}(z)>k-1\ ;\ B_{\omega^{(N)}}^{1,\ell}(w)>k-1\right)\\
%= P\big(B_{\omega^{(N)}}^{1,\ell}(z)>k-1\big)^2.
%\leq f^2(k)
&=& P\big(B_{\omega^{(N)}}^{1,\ell}(z)>k-1\big)\cdot P\big(B_{\omega^{(N)}}^{1,\ell}(w)>k-1\big).
\end{eqnarray}

Therefore, for $N$ large enough,

\begin{eqnarray*}
\var(T_{N,k})\leq
\frac{1}{(2N)^{2d}}\left[
\ell^d f(k) (2N)^d 
%+ e^{-C(\log N)^2}N^{2d}
\right]
\leq N^{-3d/4}f(k).
\end{eqnarray*}

Thus by Chebichev's inequality,
\[
P(T_{N,k}>2f(k))\leq N^{-3d/4}/f(k),
\]
and a union bound says that
\begin{eqnarray*}
&&
P\left(
\exists_{k\leq(\log\log N)^{100}}\ :\ T_{N,k}>2f(k)
\right)
\\
&\leq&
N^{-3d/4}\cdot (\log\log N)^{100}\cdot e^{C(\log\log N)^{100/3}}\\
&\leq& N^{-3d/5}.
\end{eqnarray*}

Remembering that $d\geq 2$, Borel-Cantelli now finishes the proof.

\end{proof}
\ignore{
###########################################
\begin{proof}[Proof of Claim \ref{claim:decouple}]
Fix $\ell=(\log N)^{60}$. By the definition $T^z(\omega)$, for a point $z$ an environment $\omega$,
\begin{equation}\label{eq:decoup}
T^{z}(\omega)=\sum_{k=1}^\infty k\quenchedP_\omega^z(T_1=k)
=\sum_{k=1}^{\ell-1} k\quenchedP_\omega^z(T_1=k) + \sum_{k=\ell}^\infty k\quenchedP_\omega^z(T_1=k)
\end{equation}
By Lemma \ref{lem:tail_est},
\begin{eqnarray}\label{eq:bigell}
\nonumber
E\left(\sum_{k=\ell}^\infty k\quenchedP_\omega^z(T_1=k)\right)
&=&\sum_{k=\ell}^\infty k\annealedP^z(T_1=k)\leq C\ell^3e^{-\ell^{1/3}}\\
&=&C(\log N)^{180}e^{-(\log N)^{\frac {60}3}}.
\end{eqnarray}

Write 
\[
B_\omega(z)=\sum_{k=\ell}^\infty k\quenchedP_\omega^z(T_1=k).
\]
Using Markov's inequality and a union bound, from \eqref{eq:bigell} we see that
\begin{eqnarray*}
P\left(
\exists_{z\in \Delta_N}B_\omega(z)>
\right)
\end{eqnarray*}

\end{proof}
##########################################
}

\subsection{Stationary measure}
Let $P_N$ be the uniform distribution on $\Delta_N$.
Let $\stat_N=\stat_N(\omega)$ be a stationary measure for the Markov process $\{f(Y_n)\}_{n=1}^\infty$
on $\Delta_N$ where $\{Y_n\}_{n=1}^\infty$ is the rescaled walk on $\Z^d$ under the environment $\omega^{(N)}$ and $f$ is as in \eqref{eq:deff}
(note that due to the non irreducibility of the Markov chain, there may be more than one stationary measure. In this case, $H_N$ is arbitrarily chosen among the stationary measures. Also note that by Lemma \ref{lem:quen_ntail_est}, $P$-almost surely for all large enough $N$, the process $\{Y_n\}_{n=0}^\infty$ is well defined), and let $\Phi_N=\Phi_N(\omega)=\frac{d\stat_N}{dP_N}$ be the Radon-Nikodym derivative of $\stat_N$. The main purpose of this section is the following lemma, whose proof will be completed in the next subsection.

\begin{lemma}\label{lem:lprn}
Fix $p=\frac{d}{d-1}$.
There exists a constant $C_{\ref{eq:lpbound}}$ such that for almost every $\omega$, we have that
\begin{equation}\label{eq:lpbound}
\limsup\left\{\|\Phi_N\|_{\Delta_N,p}\ :\ N=1,2,\ldots\right\}\leq C_{\ref{eq:lpbound}}.
\end{equation}
\end{lemma}

We begin with three definitions and a basic lemma, which will serve as the input for the main step.
\begin{definition}\label{def:avestep}
The {\em average step size at scale $N$}, denoted by $O_N=O_N(\omega)$ is defined to be
\begin{equation}\label{eq:avestep}
O_N:=\sqrt{E_{\stat_N}\big[(T_1\wedge N/2)^2\big]}
=\left(
\frac{1}{|\Delta_N|}\sum_{z\in \Delta_N}\Phi_N(z)\quenchedE_{\omega^{(N)}}^z\big[(T_1\wedge N/2)^2\big]
\right)^{1/2}.
\end{equation}
\end{definition}

At this point we remind the reader that $\{X_n\}$ denotes the original walk, while $\{Y_n\}$ denotes the rescaled walk. As in \cite{GZ} we define the following stopping times.
\begin{definition}\label{def:sttime_Y}
We define $S_1=S_1(N):=\inf\{n:\|Y_n-Y_0\|_\infty\geq 2N\}$ and recursively
$S_{k+1}=S_{k+1}(N):=\inf\{n>S_k:\|Y_n-Y_{S_k}\|_\infty\geq 2N\}$.
If $S_k$ is not well defined (either because the rescaled walk is not well defined or
because the walk never leaves the neighborhood of $S_{k-1}$), we set $S_k$ to infinity, as well as $S_j, j>k$.
\end{definition}
%JDDNEWREV explain why 2N
%
% Explanation just below. Noam
%
%

\ignore
{\red We need $2N$ in defs \ref{def:sttime_Y} and \ref{def:sttime_X} because the Guo-Zeitouni argument that we apply requires periodicity, and our periodicity is with a period of $2N$. - Noam}

We also define corresponding stopping times for the original walk $\{X_n\}$.

\begin{definition}\label{def:sttime_X}
We define $\Gamma_k=\Gamma_k(N)=T_{S_k}$, i.e. the time when $S_k$ occurs in the clock of the original walk.
%We define $\Gamma_1=\Gamma_1(N):=\inf\{n:\|X_n-X_0\|_\infty=N\}$ and recursively
%$\Gamma_{k+1}=\Gamma_{k+1}(N):=\inf\{n>\Gamma_k:\|X_n-X_{\Gamma_k}\|_\infty=N\}$.
\end{definition}

From the fact that $\{X_n\}$ is a martingale whose step size is one, we get the following simple estimate.

\begin{lemma}\label{lem:gamma_big}
%Let $A_N\subseteq\Omega$ be the event $A_N=\{\forall _z\in \Delta_N \quenchedP_{\omega^{(N)}}^z(T_1>\sqrt{N})\leq 1/2\}$.
There exists a constant $C$ such that for every $N$ and almost every $\omega$,
%$\omega\in A_N$,
\begin{equation}\label{eq:gamma_big}
\sum_{k=1}^\infty \quenchedP_{\omega^{(N)}}(\Gamma_k<CkN^2)<\infty.
\end{equation}
Furthermore,
\begin{equation}\label{eq:gamma_uniq}
\esssup\left\{
%{\bf 1}_{A_N}\cdot
\sum_{k=1}^\infty \quenchedP_{\omega^{(N)}}(\Gamma_k<CkN^2)
\right\}<\infty,
\end{equation}
where the essential supremum is taken w.r.t. the measure $P$ on $\omega$.
\end{lemma}

%Note that by Lemma \ref{lem:quen_ntail_est}, $\lim_{N\to\infty}P(A_N)=1$.

\begin{proof}
Note that $\Gamma_k$ is a stopping time for every $k$, and that $\|X_{\Gamma_{k+1}}-
X_{\Gamma_{k}}\|\geq 2N$. Now remember that $\{X_n\}$ is a martingale, and that the variance of its increments is 1. By Doob's inequality, there exists $C_{\ref{eq:doob}}$ such that for every balanced $\omega$ and all $k$,
\begin{equation}\label{eq:doob}
\quenchedP_\omega\left(\left.
\Gamma_{k+1}-\Gamma_k>C_{\ref{eq:doob}}N^2\ \right|
X_1,\ldots,X_{\Gamma_k}
\right)>1/2.
\end{equation}

If we now take $C=C_{\ref{eq:doob}}/4$, then by \eqref{eq:doob} and Cram\`er's Theorem we get that for every balanced $\omega$,

\begin{eqnarray}\label{eq:cramer}
\nonumber
&&\quenchedP_{\omega}(\Gamma_k<CkN^2)\\
\nonumber
&\leq&
\quenchedP_{\omega}\left(
\mbox{There exist more than $\frac{3k}{4}$ values of $n$ up to $k$ s.t. }
\Gamma_{n+1}-\Gamma_n\leq C_{\ref{eq:doob}}N^2
\right)\\
&\leq& e^{-C_{\ref{eq:cramer}}k}.
\end{eqnarray}

\eqref{eq:gamma_uniq} follows.

\end{proof}

\subsection{A bootstrap argument}
In this subsection we perform a bootstrap argument that will simultaneously control $O_N$ and prove Lemma \ref{lem:lprn}. The argument is composed of two lemmas. The first, Lemma \ref{lem:GZ},
an adaptation of Theorem 5 of \cite{GZ}, bounds $\|\Phi_N\|_{\Delta_N,p}$ in terms of $O_N$ and the second, Lemma \ref{lem:bndO}, bounds $O_N$ in terms of $\|\Phi_N\|_{\Delta_N,p}$.

We start with an a priori bound.
\begin{claim}\label{claim:ap_bound}
$P$-almost surely, $O_N\leq (\log N)^4$ for all $N$ large enough.
\end{claim}
\begin{proof}
This follows from the fact that $|\Delta_N|=N^d$ and from Lemma \ref{lem:ntail_est}, the same way 
Lemma \ref{lem:quen_tail_est} is proven.
\end{proof}

\begin{lemma}\label{lem:GZ}
$P$-almost surely, there exists a constant $C_1$ such that for every $N$ large enough,
\[
\|\Phi_N\|_{\Delta_N,p}<C_1\cdot O_N,
\]
where, as before, $p=\frac{d}{d-1}$.
\end{lemma}

\begin{lemma}\label{lem:bndO}
$P$-almost surely, there exists a constant $C_2$ such that for every $N$ large enough and every $k<(\log N)^5$, if 
\[
\|\Phi_N\|_{\Delta_N,p}<k
\]
then
\[
O_N<C_2(\log k)^4.
\]
\end{lemma}

\begin{proof}[Proof of Lemma \ref{lem:lprn}]
The combination of Claim \ref{claim:ap_bound} and Lemmas \ref{lem:GZ} and \ref{lem:bndO} yields that for all $N$ large enough, $O_N\leq C_1C_2(\log O_N)^4$, and therefore $\sup\{O_N: N=1,2,\ldots\}<\infty$. Another application of Lemma \ref{lem:GZ} yields \eqref{eq:lpbound}.
\end{proof}

\begin{proof}[Proof of Lemma \ref{lem:bndO}]
Let $j=(\log k)^4$, and let $f(z)=\quenchedE_{\omega^{(N)}}^z\big[(T_1 \wedge N/2)^2\big]$. Then
\begin{eqnarray*}
O_N^2=\frac{1}{N^d}\sum_{z\in \Delta_N}\Phi_N(z)f(z) &=&\\
        \frac{1}{N^d}\sum_{z\in \Delta_N}\Phi_N(z)f(z){\bf 1}_{f(z)\leq j}
&+& \frac{1}{N^d}\sum_{z\in \Delta_N}\Phi_N(z)f(z){\bf 1}_{f(z)> j}\\
\leq j
&+& \|\Phi_N\|_{\Delta_N,p} \|f(z){\bf 1}_{f(z)> j}\|_{\Delta_N,d}\\
\leq j+ke^{-Cj^{1/3}}&\leq& 2j
\end{eqnarray*}
where the one before last inequality follows from H\"older's inequality, Lemma \ref{lem:emp_dist} and the assumption that $\|\Phi_N\|_{\Delta_N,p}<k$, and the last inequality follows from the fact that 
$j^{1/3}>\log k$. 

\end{proof}

\begin{proof}[Proof of Lemma \ref{lem:GZ}]
The argument is based on the proof of Theorem 5 in \cite{GZ}. Let $h:\Delta_{2N}\to\R^+$ be a test function. 
We extend $h$ to the entire $Z^d$ by $h(x):=h(x\mod 2N)$.
%%
%JDDNEWREV this is very confusing, that we suddently go to the box of length 2N while prevviously length $N$

We remember that $\{Y_\cdot\}$ is the rescaled walk, and extend $\Phi_N$ to $\Delta_{2N}$ by
$\Phi_N(x):=\Phi_N(f(x))$ for $f$ as in \eqref{eq:deff}. The extended $\Phi_N$ is the Radon-Nikodym derivative of the measure $\tH$ defined as $\tH(x)=\frac{1}{2^d}H_N(f(x))$. 
%JDDNEWREV below we use $\tH_N$ here no N!
Note that $\tH$ is stationary with respect to the (periodized) random walk on $\Delta_{2N}$.
% and when applying $h$ to a point $z$ outside $\Delta_N$ we take $h$ of $z\mod 2N$.
Then
\begin{eqnarray}\label{eq:GZ}
\nonumber
&&\frac{1}{(2N)^d}\sum_{z\in \Delta_{2N}}\Phi_N(z)h(z)\\
\nonumber
&=&\frac{O_N}{N^2}\sum_{z\in \Delta_{2N}} \frac{\Phi_N(z)}{(2N)^d}
\sum_{j=0}^\infty\quenchedE_{\omega^{(N)}}^z\left(1-\frac{O_N}{N^2}\right)^jh(Y_j)\\
\nonumber
&=&\frac{O_N}{N^2}\sum_{z\in \Delta_{2N}} \frac{\Phi_N(z)}{(2N)^d}\sum_{m=0}^\infty
\sum_{j=S_m}^{S_{m+1}-1}\quenchedE_{\omega^{(N)}}^z\left(1-\frac{O_N}{N^2}\right)^jh(Y_j)\\
\\
\nonumber
&\leq&\frac{O_N}{N^2}\sum_{z\in \Delta_{2N}} \frac{\Phi_N(z)}{(2N)^d}\sum_{m=0}^\infty
\quenchedE_{\omega^{(N)}}^z\left(1-\frac{O_N}{N^2}\right)^{S_m}
\quenchedE_{\omega^{(N)}}^{Y_{S_m}} \sum_{j=0}^{S_{m+1}-S_m-1} h(Y_j)\\
\nonumber
&\leq&\frac{O_N}{N^2}
\left(\max_{z\in \Delta_{2N}}\quenchedE_{\omega^{(N)}}^z \sum_{j=0}^{S_1-1} h(Y_j)\right)
\left(\sum_{z\in \Delta_{2N}} \frac{\Phi_N(z)}{(2N)^d}\sum_{m=0}^\infty
\quenchedE_{\omega^{(N)}}^z\left(1-\frac{O_N}{N^2}\right)^{S_m}\right)
\end{eqnarray}

We use the following claim, whose proof will be given at the end of the proof of the lemma.
\begin{claim}\label{claim:from_max}
\begin{equation}\label{eq:from_max_princ}
\max_{z\in \Delta_{2N}}\quenchedE_{\omega^{(N)}}^z \sum_{j=0}^{S_1-1} h(Y_j)
\leq CN^2\|h\|_{\Delta_{2N},d}.
\end{equation}
\end{claim}

We now estimate the remaining term, namely
\[
\sum_{z\in \Delta_{2N}} \frac{\Phi_N(z)}{(2N)^d}\sum_{m=0}^\infty
\quenchedE_{\omega^{(N)}}^z\left(1-\frac{O_N}{N^2}\right)^{S_m}.
\]
Note that this is 
\[
\sum_{m=0}^\infty E_{\tH_N}\left(1-\frac{O_N}{N^2}\right)^{S_m}
=\sum_{m=0}^\infty E_{H_N}\left(1-\frac{O_N}{N^2}\right)^{S_m},
\]
and that $H_N$ is a stationary distribution for $\{f(Y_\cdot)\}$. In particular, the sequence $\{T_k-T_{k-1}\}$ is stationary under $H_N$, and $E_{H_N}\big[(T_k-T_{k-1})^2\big]=O_N^2$ for every $k$.

Now, for a given $m>0$,
\begin{eqnarray}\label{eq:sm_1}
\nonumber
&& E_{H_N}\left(1-\frac{O_N}{N^2}\right)^{S_m} \\
\nonumber&\leq& \left(1-\frac{O_N}{N^2}\right)^{\frac{N^2}{O_N}(\log m)^4} + P_{H_N}\left(S_m<\frac{N^2}{O_N}(\log m)^4\right)\\
&\leq& 2e^{-(\log m)^4}+P_{H_N}\left(S_m<\frac{N^2}{O_N}(\log m)^4\right).
\end{eqnarray}
Let $C_{\ref{lem:gamma_big}}$ be the constant from Lemma \ref{lem:gamma_big}. Then

\begin{eqnarray}\label{eq:sm_2}
&&P_{H_N}\left(S_m<\frac{N^2}{O_N}(\log m)^4\right) \\
\nonumber
&\leq& P_{H_N}\left(\Gamma_m<C_{\ref{lem:gamma_big}}mN^2\right)
 + P_{H_N}\left(S_m<\frac{N^2}{O_N}(\log m)^4 \ ;\ \Gamma_m\geq C_{\ref{lem:gamma_big}}mN^2\right)
\end{eqnarray}

Lemma \ref{lem:gamma_big} takes care of the first summand, so all we have left to do is to control the second summand. By Markov's inequality,
\begin{eqnarray}\label{eq:markov}
\nonumber
&&P_{H_N}\left(S_m<\frac{N^2}{O_N}(\log m)^4 \ ;\ \Gamma_m\geq C_{\ref{lem:gamma_big}}mN^2\right)\\
\nonumber
&\leq& 
%JDDNEW introduce [N^2/O_N\log...    ]
P_{H_N}\left(T_{\left[\frac{N^2}{O_N}(\log m)^4\right]}\geq C_{\ref{lem:gamma_big}}mN^2 \right)\\
&\leq & \frac {E_{H_N}\left[\left(T_{\left[\frac{N^2}{O_N}(\log m)^4\right]}\right)^2\right]}
{C_{\ref{lem:gamma_big}}^2m^2N^4}
\end{eqnarray}

and

\begin{eqnarray*}
&&E_{H_N}\left[\left(T_{\left[\frac{N^2}{O_N}(\log m)^4\right]}\right)^2\right]
= E_{H_N}\left[\left(
\sum_{i=1}^{\left[\frac{N^2}{O_N}(\log m)^4\right]}T_i-T_{i-1}
\right)^2\right]\\
&=& \sum_{i=1}^{\left[\frac{N^2}{O_N}(\log m)^4\right]} \sum_{j=1}^{\left[\frac{N^2}{O_N}(\log m)^4\right]}
E_{H_N}[(T_i-T_{i-1})(T_j-T_{j-1})]\\
&\leq& \frac{N^4}{O_N^2}(\log m)^8\cdot O_N^2 = N^4(\log m)^8
\end{eqnarray*}

Substituting in \eqref{eq:markov}, we get that
\[
P_{H_N}\left(S_m<\frac{N^2}{O_N}(\log m)^4 \ ;\ \Gamma_m\geq C_{\ref{lem:gamma_big}}mN^2\right) \leq \frac{(\log m)^8}{m^2 C_{\ref{lem:gamma_big}}^2}
\]

and combined with \eqref{eq:sm_1}, \eqref{eq:sm_2} and Lemma \ref{lem:gamma_big}, we get that for every test function $h$,
\[
\frac{1}{(2N)^d}\sum_{z\in \Delta_{2N}}\Phi_N(z)h(z)\leq C_1\|h\|_{\Delta_{2N},d}O_N.
\]

The duality of $L^d$ and $L^p$ now gives that $\|\Phi_N\|_{\Delta_N,p}=\|\Phi_N\|_{\Delta_{2N},p}\leq C_1O_N.$
%JDDNEWREV why equality, trivial ?
%
% The equality is because the $2d$ copies of $\Delta_N$ in $\Delta_{2N}$ distribute the function $\Phi$ completely
% equally.  Noam
%
\end{proof}
\begin{proof}[Proof of Claim \ref{claim:from_max}]
We need to show \eqref{eq:from_max_princ}.
We first estimate 
\[
\max_{z\in \Delta_{2N}}\quenchedE_{\omega^{(N)}}^z \sum_{j=0}^{S_1-1} h(Y^{(N)}_j)
\]
Where the walk $Y^{(N)}_j$ is defined by $Y^{(N)}_j=X_{T^{(N)}_j}$, with $T_1^{(N)}:=\min(T_1,N/2)$ and 
\[
T_{k+1}^{(N)}:=\min\big\{\{t>T_k^{(N)}\ :\ \{\alpha(T_k^{(N)}+1),\ldots,\alpha(t)\}=\{1,\ldots,d\}\}
\cup\{T_k^{(N)}+N/2\}\big\}.
\]

We fix $z\in \Delta_{2N}$,
and define the stopping time $T^{z}=\min\{j\ :\ Y^{(N)}_j\notin z+\Delta_{2N}\}$ and the function 
\[
f^z(x)=
\quenchedE_{\omega^{(N)}}^x \sum_{j=0}^{T^z-1} h(Y^{(N)}_j).
\]
Then $L^{(N)}f^z=h$. Almost surely, for all $N$ large enough, 
Condition \eqref{eq:condformax} with $k=N/2$ is satisfied by Lemma \ref{lem:quen_ntail_est}, and
therefore by Theorem \ref{thm:max_princ},
\[
\quenchedE_{\omega^{(N)}}^z \sum_{j=0}^{S_1-1} h(Y^{(N)}_j)=f^z(z)\leq CN^2\|h\|_{\Delta_{2N},d}.
\]

Therefore, all we need is to control 
\[
\quenchedE_{\omega^{(N)}}^z 
\left[
\left|\sum_{j=0}^{S_1-1} h(Y^{(N)}_j)
-
 \sum_{j=0}^{S_1-1} h(Y_j)
\right|\right].
\]

Now,
\[
\quenchedE_{\omega^{(N)}}^z 
\left[
\left|\sum_{j=0}^{S_1-1} h(Y^{(N)}_j)
-
 \sum_{j=0}^{S_1-1} h(Y_j)
\right|^2\right]
\leq CN^4\max_{z\in \Delta_{2N}}h^2(z)
\]
and 
\[
\quenchedP_{\omega^{(N)}}^z 
\left[
\left|\sum_{j=0}^{S_1-1} h(Y^{(N)}_j)
-
 \sum_{j=0}^{S_1-1} h(Y_j)
\right|\neq 0\right]\leq N^2e^{-cN^{1/3}}.
\]

>From Cauchy-Schwarz, we see that 
\begin{equation}\label{eq:cs}
\quenchedE_{\omega^{(N)}}^z 
\left[
\left|\sum_{j=0}^{S_1-1} h(Y^{(N)}_j)
-
 \sum_{j=0}^{S_1-1} h(Y_j)
\right|\right]\leq CN^4e^{-cN^{1/3}}\max_{z\in \Delta_{2N}}h(z).
\end{equation}
%JDDNEW gat to get
Noting that the size of the space $\Delta_{2N}$ is $(2N)^d$, we get that 
\[
\max_{z\in \Delta_{2N}}h(z)=\|h\|_{\Delta_{2N},\infty}\leq (2N)^d \|h\|_{\Delta_N,d}.
\]
%JDDNEWREV why the last inequality, do you mean $\Delta_{2N}$?
%
% Changed to $\Delta_{2N}$, as you suggested.
% The last inequality follows from the finiteness of the space $\Delta_{2N}$ - on finite dimensional spaces all
% norms are equivalent and you can write inequalities comparing them.   Noam
%
%

With \eqref{eq:cs} we are now done.

\end{proof}

\subsection{A stationary measure for the original random walk on $\Delta_N$}\label{sec:non-rescaled}
Fix $p'$ to be strictly between $1$ and $p$. In Lemma \ref{lem:lprn} we controlled the $L^p$ norm of a stationary measure w.r.t. the rescaled random walk. We now use Lemma \ref{lem:lprn} to control the $L^{p'}$ norm of a stationary measure w.r.t. the original random walk.

\begin{lemma}\label{lem:lporig}
There exists $C$ such that $P$-almost surely for all $N$ large enough, every probability measure $Q_N$ which is stationary with respect to the original reflected random walk on $\Delta_N$ satisfies 
\[
\left\|
\frac{dQ_N}{dP_N}
\right\|_{\Delta_N,p'}<C.
\]
\end{lemma}

\begin{proof}
First note that due to the convexity of the norm $\|\cdot\|_{\Delta_N,p'}$, we may assume without loss of generality that the measure $Q_N$ is ergodic. Then the random walk is irreducible on $\supp Q_N$. It is also clear that if the random walk starts at a point in $\supp Q_N$, it will stay in $\supp Q_N$ forever. Therefore, there exists a measure $H_N$ which is supported on (a subset of) $\supp Q_N$ and is stationary with respect to the rescaled random walk.

Now consider the following random walk $\{X_n\}_{n=0}^\infty$ on $\Delta_N$: the initial point $X_0$ is determined according to the distribution $H_N$, and the walk continues according to the quenched kernel $\omega$ on $\Delta_N$, reflected at the boundary. 

For $i=0,\ldots,$ we define the measure (not a probability measure) $F_i$ on $\Delta_N$ by
\[
F_i(x)=\sum_{z\in\Delta_N}H_N(z)\quenchedP_\omega^z(X_i=x\ ; \ T_1>i).
\]

\begin{claim}\label{claim:sumfi}
For $P$-almost every $\omega$ and all $N$ large enough,
the sum
\[
\sum_{i=0}^\infty F_i
\]
converges to a finite measure $F$. Furthermore, $\|F\|_1=E_{H_N}(T_1)$ and $F$ is stationary w.r.t. the (original) random walk.
\end{claim}

Since the random walk is irreducible on $\supp Q_N$, there is a unique stationary measure for the original random walk, and therefore $Q_N=F/E_{H_N}(T_1)$. As $E_{H_N}(T_1)>1$, we get that
\begin{equation*}
\left\|
\frac{dQ_N}{dP_N}
\right\|_{\Delta_N,p'}
\leq
\sum_{i=0}^\infty
\left\|
\frac{dF_i}{dP_N}
\right\|_{\Delta_N,p'}.
\end{equation*}
Therefore, we want to estimate 
$
\left\|
\frac{dF_i}{dP_N}
\right\|_{\Delta_N,p'}
$
for given $i$.

%
%$F_i(x)=G_i(x)\cdot B_i(x)$ where
%\[
%G_i(x) = \sum_{z\in\Delta_N}H_N(z)\quenchedP_\omega^z(X_i=x)
%\]
%and
%\[
%B_i(x) = \sum_{z\in\Delta_N}H_N(z)\quenchedP_\omega^z(T_1>i\ |\ X_i=x).
%\]

We first estimate  
$
\left\|
\frac{dF_i}{dP_N}
\right\|_{\Delta_N,p}.
$
Note that 
\[
F_i(x)\leq G_i(x) := \sum_{z\in\Delta_N}H_N(z)\quenchedP_\omega^z(X_i=x),
\]
%and
%let $B_i(x)=F_i(x)/G_i(x)$.

\begin{eqnarray*}
G_i(x) &=& \sum_{z\in\Delta_N\ :\ |z-x|\leq i}H_N(z)\quenchedP_\omega^z(X_i=x)\\
&\leq& \sum_{z\in\Delta_N\ :\ |z-x|\leq i}H_N(z).
\end{eqnarray*}

Therefore
\begin{eqnarray*}
\left(G_i(x)\right)^p
\leq (2i+1)^{d(p-1)}\sum_{z\in\Delta_N\ :\ |z-x|\leq i}(H_N(z))^p,
\end{eqnarray*}

so
%\begin{eqnarray}\label{eq:normgi}
%\sum_{x\in\Delta_N}\left(\frac{G_i(x)}{P_N(x)}\right)^p
%\leq
%(2i+1)^{dp}\|\frac{dH_N}{dP_N}\|_{\Delta_N,p}^p
%\end{eqnarray}

\begin{eqnarray}\label{eq:normgi}
\left\|
\frac{dF_i}{dP_N}
\right\|_{\Delta_N,p}
\leq
\left\|\frac{dG_i}{dP_N}\right\|_{\Delta_N,p}
%\sum_{x\in\Delta_N}\left(\frac{G_i(x)}{P_N(x)}\right)^p
\leq
(2i+1)^{d}\left\|\frac{dH_N}{dP_N}\right\|_{\Delta_N,p}.
\end{eqnarray}

Let $p''$ be such that $\frac{1}{p''}+\frac{p'}{p}=1$.

We also want to estimate
$
\left\|
\frac{dF_i}{dP_N}
\right\|_{\Delta_N,1}.
$

\begin{eqnarray*}
\left\|
\frac{dF_i}{dP_N}
\right\|_{\Delta_N,1} = \sum_{x\in\Delta_N}F_i(x)
&=&\sum_{x\in\Delta_N}\sum_{z\in\Delta_N}H_N(z)\quenchedP_\omega^z(X_i=x\ ;\ T_1>i)\\
&=&\sum_{z\in\Delta_N}H_N(z)\quenchedP_\omega^z(T_1>i) \leq e^{-Ci^{1/3}},
\end{eqnarray*}
%JDDNEWREV give a ref. for last inequality
where the last inequality follows from Lemma \ref{lem:ntail_est}.
%{\red \LARGE GGGGG}

Now let
\[
Y(x)={\bf 1}_{\left\{\frac{F_i(x)}{P_N(x)}>1\right\}}
\]
and let $X_1(x)=Y(x)\cdot\frac{F_i(x)}{P_N(x)}$ and $X_2(x)=(1-Y(x))\cdot\frac{F_i(x)}{P_N(x)}$. Since $\frac{dF_i}{dP_N}=X_1+X_2$, we need to estimate $\|X_1\|_{\Delta_N,p'}$ and $\|X_2\|_{\Delta_N,p'}$.

Note that $X_2(x)\leq 1$ for every $x$, and therefore, $X_2^{p'}(x)\leq X_2(x)$. Therefore, 
$\|X_2\|_{\Delta_N,p'}\leq \|X_2\|_{\Delta_N,1}^{1/p'} \leq \exp\big(-(C/p')i^{1/3}\big).$

Note also that $Y(x)\in\{0,1\}$ and thus $Y^{p''}=Y$, so, using Markov's inequality,
%JDDNEWREV c to -C 
\[
\|Y\|_{\Delta_N,p''}^{p''}=\|Y\|_{\Delta_N,1}\leq\left\|
\frac{dF_i}{dP_N}
\right\|_{\Delta_N,1}\leq \exp\big(-Ci^{1/3}\big),
\]

so $\|Y\|_{\Delta_N,p''}\leq \exp\big(-(C/p'')i^{1/3}\big).$
Then by H\"older's inequality,
\[
\|X_1\|_{p'}^{p'}\leq \left\|\frac{dF_i}{dP_N}
\right\|_{\Delta_N,p}^{p'}\cdot \|Y\|_{p''}
\leq \|H_N\|_p^{p'}\cdot (2i+1)^{dp'}  \exp\big(-(C/p'')i^{1/3}\big).
\]

We get that for appropriate constants $C_1$ and $C_2$,

\[
\left\|\frac{dF_i}{dP_N}\right\|_{p'}\leq C_1 (2i+1)^{d} \exp\big(-C_2i^{1/3}\big) \cdot \|H_N\|_p.
\]

The lemma now follows from Lemma \ref{lem:lprn} and the fact that
\[
\sum_{i=0}^\infty C_1 (2i+1)^{d} \exp\big(-C_2i^{1/3}\big) < \infty.
\]

\end{proof}

\begin{proof}[Proof of Claim \ref{claim:sumfi}]
First of all,
\begin{eqnarray*}
\|F_i\|_1= \sum_{x\in\Delta_N}F_i(x)
&=&\sum_{x\in\Delta_N}\sum_{z\in\Delta_N}H_N(z)\quenchedP_\omega^z(X_i=x\ ; \ T_1>i)\\
&=&\sum_{z\in\Delta_N}\sum_{x\in\Delta_N}H_N(z)\quenchedP_\omega^z(X_i=x\ ; \ T_1>i)\\
&=&\sum_{z\in\Delta_N}H_N(z)\quenchedP_\omega^z(T_1>i) = P_{H_N}(T_1>i).
\end{eqnarray*}
%JDDNEWREV write that this is \le e^{-Ci^{1/3}}
Therefore $\sum_{i=1}^\infty F_i$ converges and $\|F\|_1=E_{H_N}(T_1)$.

To show stationarity, we do the following calculation. Fix $x\in\Delta_N$.
\begin{eqnarray*}
\sum_{y\in\Delta_N}F(y)\quenchedP_\omega^y(X_1=x)
&=&\sum_{i=0}^\infty\sum_{y\in\Delta_N}F_i(y)\quenchedP_\omega^y(X_1=x)\\
=\sum_{i=0}^\infty\sum_{z\in\Delta_N}H_N(z)\quenchedP_\omega^z(X_{i+1}=x\ ; \ T_1>i)
&=&\sum_{j=1}^\infty\sum_{z\in\Delta_N}H_N(z)\quenchedP_\omega^z(X_{j}=x\ ; \ T_1\geq j)\\
=\sum_{j=1}^\infty\sum_{z\in\Delta_N}H_N(z)\quenchedP_\omega^z(X_{j}=x\ ; \ T_1> j)
&+&\sum_{j=1}^\infty\sum_{z\in\Delta_N}H_N(z)\quenchedP_\omega^z(X_{j}=x\ ; \ T_1= j)\\
=\sum_{j=1}^\infty F_j(x)+\sum_{z\in\Delta_N}H_N(z)\quenchedP_\omega^z(X_{T_1}=x)
&=&\sum_{j=1}^\infty F_j(x) + H_N(x) = F(x),
\end{eqnarray*}
where in the one before last step we used the stationarity of $H_N$ with respect to the rescaled random walk.
\end{proof}

We get a useful corollary.

\begin{corollary}\label{cor:phi}
There exists $\Phi>0$ which depends only on $P$ such that $P$-almost surely for all large enough $N$, every stationary measure $Q_N$ with respect to the reflected random walk in $\Delta_N$ under the environment $\omega$ satisfies
$|\supp Q_N|\geq\Phi|\Delta_N|$. 
\end{corollary}

\begin{proof}
By Lemma \ref{lem:lporig},
\[
\left\|
\frac{dQ_N}{dP_N}
\right\|_{\Delta_N,p'}<C'.
\]
Now, by H\"older's inequality, for $p''$ such that $1/p'+1/p''=1$,
\begin{eqnarray*}
1
=\left\|
\frac{dQ_N}{dP_N}
\right\|_{\Delta_N,1}
\leq
\left\|
\frac{dQ_N}{dP_N}
\right\|_{\Delta_N,p'}
\cdot
\left\|
{\bf 1}_{\supp Q_N}
\right\|_{\Delta_N,p''}
\leq C'\left(
\frac{|\supp Q_N|}{|\Delta_N|}
\right)^{1/p''}.
\end{eqnarray*}

Therefore,
\[
\frac{|\supp Q_N|}{|\Delta_N|}\leq\frac1{{C'}^{p''}}=:\Phi.
\]

\end{proof}

\subsection{Existence of an invariant measure.}\label{sec:stnons}
Identically to \cite{GZ} and \cite{Lawler82}, from Lemma \ref{lem:lporig} we can prove that there exists a measure $Q$ on $\Omega$ such that $Q\ll P$ and $Q$ is stationary with respect to the random walk  viewed from the point of view of the particle.
\ifarxiv
In order to keep this paper self-contained, we state and prove it as Proposition \ref{prop:exist} below.
\else
We state it explicitly as Proposition \ref{prop:exist} below. The proof of this proposition is in the arxiv version of the paper.
\fi

Once we established $Q$, using Feller-Lindeberg's central limit theorem, see e.g. \cite{durrett}, we get the following fact.

\begin{fact}\label{fact:5.1}
If in addition $Q$ is ergodic, then $Q$ almost surely the quenched law $\quenchedP_\omega^0$
satisfies an invariance principle with a non-random diagonal, non-degenerate diffusion matrix.
\end{fact}

\begin{proof}[Proof that the matrix is diagonal and non-degenerate]
The proof that the matrix is diagonal is easy: For every balanced $\omega\in\Omega$, every $1\leq i\neq j\leq d$ and every two times $n,m$,
\[
\quenchedE_\omega\big[\langle X_n-X_{n-1},e_i\rangle\cdot\langle X_m-X_{m-1},e_j\rangle\big]=0
\]
and therefore the covariance matrix of $X_n$ is diagonal for every $n$, and therefore the diffusion matrix is diagonal.
Let $M$ be the diffusion matrix. Since it is diagonal, in order to see that it is non-degenerate, all we need is to show that $M_{i,i}\neq 0$ for every $i$. Now, by the stationarity and ergodicity of $Q$,
\begin{eqnarray*}
M_{i,i}&=&\lim_{n\to\infty}\frac 1n\sum_{j=1}^n{\bf 1}_{\langle X_j-X_{j-1},e_i\rangle\neq 0}\\
&\geq&
\lim_{n\to\infty}\frac 1n\sum_{j=1}^n{\bf 1}_{\exists_k\mbox{ s.t }j=T_k}
=\frac 1 {E_Q(T_1)}>0.
\end{eqnarray*}
\end{proof}

Note that even though $Q\ll P$, in the non-elliptic case it is not necessarily the case that $P\ll Q$, as is illustrated in Figure \ref{fig:noteqviv}.

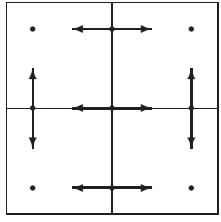
\begin{figure}[t]
\begin{center}
{%\blue
\begin{picture}(80,80)
\linethickness{0.05pt}
\put(0,0){\line(0,1){80}}
\put(0,0){\line(1,0){80}}
\put(80,80){\line(0,-1){80}}
\put(80,80){\line(-1,0){80}}
\linethickness{0.2pt}
\put(40,0){\line(0,1){80}}
\put(0,40){\line(1,0){80}}
\linethickness{0.4pt}
\multiput(40,10)(0,30){3}{\vector(1,0){15}}
\multiput(40,10)(0,30){3}{\vector(-1,0){15}}
\multiput(10,40)(60,0){2}{\vector(0,1){15}}
\multiput(10,40)(60,0){2}{\vector(0,-1){15}}
\multiput(10,10)(0,30){3}{\circle*{1.5}}
\multiput(40,10)(0,30){3}{\circle*{1.5}}
\multiput(70,10)(0,30){3}{\circle*{1.5}}
\end{picture}
}
\caption{\sl 
A configuration that has a positive $P$-measure, but zero $Q$-measure. The $Q$-measure is zero because the configuration presented here cannot occur at the second step, and $Q$ is stationary (i.e. the second step has the same distribution as the first step).
}
\label{fig:noteqviv}
\end{center}
\end{figure}

In Section \ref{sec:high_d} we show how the two remaining problems (i.e. the question of ergodicity and the fact that the measures are not equivalent) are dealt with.% in dimension 2. In Section \ref{sec:high_d} we solve these problems in all dimensions. The (simpler) two-dimensional method is not extendible to higher dimensions.

\begin{proposition}\label{prop:exist}
There exists a probability measure $Q$ on $\Omega$ such that
\begin{enumerate}
\item $Q\ll P$.
\item $Q$ is invariant w.r.t. the point of view of the particle.
\end{enumerate}
\end{proposition}

\ifarxiv
\else
\ignore{
\fi

\begin{proof}
Fix $\omega\in\Omega$, and define $Q_N$ to be an (arbitrarily chosen) invariant measure for the reflected random under the environment $\omega$ in $\Delta_N$, and let $P_N$ be the uniform measure on $\Delta_N$.
For every $N$ we define the measure $Q^{(N)}$ on $\Omega$ to be
\[
Q^{(N)}=\sum_{z\in\Delta_N}Q_N(z)\delta_{\tau_{-z}(\omega)}
\]
and the measure $P^{(N)}$ to be 
\[
P^{(N)}=\sum_{z\in\Delta_N}P_N(z)\delta_{\tau_{-z}(\omega)}.
\]

By compactness, there exists a subsequence $(Q^{(N_k)})$ which converges weakly to a probability measure $Q$ on $\Omega$. It is easy to show that $Q$ is invariant w.r.t. the point of view of the particle for $P$-almost every $\omega$, so we now show that $Q\ll P$ for $P$-almost every $\omega$, using the fact that for $P$-a.e. $\omega$, for every $j$,
\[
\lim_{N\to\infty}Q_N(\{z:\exists_{x\in\partial\Delta_N}\mbox{ s.t. } \|z-x\|\leq j\})=0.
%\lim_{N\to\infty}Q_N(\{z:\exists_{x\in\Z^d\setminus\Delta_N}\mbox{ s.t. } \|z-x\|\leq j\})=0.
\]

Note that for almost every $\omega$, the sequence $P^{(N)}$ converges to $P$, and there exists $C$ such that $\big\|\frac{dQ_N}{dP_N}\big\|_{p'}<C$ for every $N$. From this we get immediately that for every $N$,
\begin{equation}\label{eq:bndrn}
\big\|\frac{dQ^{(N)}}{dP^{(N)}}\big\|_{p'}<C. 
\end{equation}
Assume for contradiction that $Q\not\ll P$. Then there exists $A\subseteq\Omega$ such that $\alpha:=Q(A)>0$ and $P(A)=0$. For every $\epsilon$ we can find an $A_\epsilon\subseteq\Omega$ which is determined by finitely many coordinates, and such that $Q(A_\epsilon)>\alpha/2$ and $P(A_\epsilon)\leq\epsilon$. Then for all $k$ large enough, $Q^{(N_k)}(A_\epsilon)>\alpha/4$ and $P^{(N_k)}(A_\epsilon)\leq 2\epsilon$, and therefore 
\[
\left\|\frac{dQ^{(N)}}{dP^{(N)}}\right\|^{p'}_{p'}
\geq P^{(N_k)}(A_\epsilon)\left(\frac{Q^{(N_k)}(A_\epsilon)}{P^{(N_k)}(A_\epsilon)}\right)^{p'}
\geq \left(\frac\alpha 4\right)^{p'}(2\epsilon)^{1-p'}
\]

For $\epsilon$ small enough this is in contradiction with \eqref{eq:bndrn}.

\end{proof}

\ifarxiv
\else
}
\fi

\ignore{

%\section{A coupling argument in two dimensions}\label{sec:coup2d}

The main statement in this section is the coupling statement in Lemma \ref{lem:coup} below. We will then apply Lemma \ref{lem:coup} to prove Lemma \ref{lem:ergod} which says that $Q$ is ergodic, and Lemma \ref{lem:invprin} which says that an invariance principle holds $P$-a.s.
Note that Lemma \ref{lem:invprin} supplies a proof for Theorem \ref{thm:main} in dimension 2.

\begin{lemma}\label{lem:coup} Let $d=2$. Let $x$ and $y$ be two vertices in $\Z^2$ which have the same parity (i.e. $(-1)^{x^{(1)}+x^{(2)}}=(-1)^{y^{(1)}+y^{(2)}}$). Then $P$-a.s. there exists a coupling $\mu_{\omega}^{x,y}$ of two random walks $\{X_n\}$ and $\{Y_n\}$ with marginal distributions (resp.) $\quenchedP_\omega^x$ and $\quenchedP_\omega^y$ s.t.
\[
\mu_{\omega}^{x,y}
\left[
\exists_T\ \mbox{s.t.}\ \forall_{t>T}\ X_t=Y_t
\right]=1.
\]
\end{lemma}

\begin{lemma}\label{lem:ergod}
Let $d=2.$ Then the measure $Q$ is ergodic.
\end{lemma}

\begin{lemma}\label{lem:invprin}
Let $d=2.$ Then
$P$-a.s. the law $\quenchedP_\omega^0$ satisfies an invariance principle.
\end{lemma}

\begin{remark}
Using the invariance principle and the existence of an invariant measure $Q$, by an argument of Kesten we can prove that the random walk is recurrent on $\supp_\omega Q$. It is clearly not recurrent outside $\supp_\omega Q$.
\end{remark}

\begin{proof}[Proof of Lemma \ref{lem:coup}]
We define the coupling as follows. First we run $X_n$ and $Y_n$ independently of each other, until the stopping time $S$ defined as $S=\min\{n>0\ :\ X^{(1)}_n+X^{(2)}_n=Y^{(1)}_n+Y^{(2)}_n\}$. The stopping time $S$ is almost surely finite because $\{W_\cdot=X^{(1)}_\cdot+X^{(2)}_\cdot\}$
and $\{Z_\cdot=Y^{(1)}_\cdot+Y^{(2)}_\cdot\}$ are independent one-dimensional simple random walks having the same parity.

After the stopping time $S$, and until the stopping time $T=\min\{n>0 : X_n = Y_n \}$ for each time $n\geq S$, we sample three random variables as follows: The variable $H_n^{(X)}$ takes the value $1$ with probability $\omega(X_n,e_1)+\omega(X_n,-e_1)$ and the value $2$ with probability $\omega(X_n,e_2)+\omega(X_n,-e_2)$.
The variable $H_n^{(Y)}$ is defined equivalently for $Y_n$.
 $U_n$ takes the value $+1$ with probability 0.5 and the value $-1$ with the same probability. Conditioned on $X_n$, $Y_n$ and $\omega$, the variables $H_n^{(X)}, H_n^{(Y)}$ and $U_n$ are independent of each other and of $X_k, Y_k, k\leq n-1$.

Let
\[
e(i)=\left\{
\begin{array}{ll}
e_1 & \mbox{if } i=1\\
e_2 & \mbox{if } i=2
\end{array}
\right..
\]
Then we take $X_{n+1}=X_n+U_n\cdot e\big(H_n^{(X)}\big)$ and $Y_{n+1}=Y_n+U_n\cdot e\big(H_n^{(Y)}\big)$. 

>From time $T$ and on, the two walks move together.

It is very easy to verify that this is indeed a coupling. What's left to be shown is that for almost every $\omega$, with probability $1$ the stopping time $T$ is finite.

To this end, note that by the construction of the coupling, for every $n\geq S$ we have that $W_n=Z_n$. We thus need to show that almost surely there exists $n>S$ such that
$\bar W_n=\bar Z_n$ where $\bar W_n=X_n^{(1)}-X_n^{(2)}$ and $\bar Z_n=Y_n^{(1)}-Y_n^{(2)}$. Note that $\bar W_n-\bar Z_n$ is (up to a factor of 2) is a time changed one-dimensional simple random walk. Thus we need to show that conditioned on the event $T=\infty$, the random walk makes infinitely many steps. Let $n_1, n_2, \ldots$ be the sequence of times in which $W_n=S_n=1+ \max\{W_k\ :\ k\leq n-1\}$. Clearly there are infinitely many such times, and in such times both $X_n$ and $Y_n$ are points that had never been visited before. Thus, in the annealed sense, $H_{n_k}^{(X)}, H_{n_k}^{(Y)}, k=1,2,\ldots$ are all independent, and thus $H_{n_k}^{(X)}\neq H_{n_k}^{(Y)}$ infinitely many times and thus the random walk moves infinitely often.

\end{proof}

\begin{proof}[Proof of Lemma \ref{lem:ergod}]
Assume for contradiction that $Q$ is not ergodic.
Let $G=\supp Q = \{\omega\in\Omega\ :\ \frac{dQ}{dP}(\omega)>0\}$ where the derivative is a Radon-Nikodym derivative. 
Then there exists a set $A\subseteq\Omega$ such that $0<Q(A)<1$ and such that the measure $Q_A$ defined by $Q_A(U)=Q(U\cap A)/Q(A)$ is stationary w.r.t. the random walk viewed from the point of view of the particle. Let $B=G\setminus A$. Then $0<Q(B)<1$ and $Q_B$ is stationary w.r.t. the random walk viewed from the point of view of the particle. For $\omega\in\Omega$, we define 
\[
A_\omega= \{z\in\Z^2\ :\ \tau_{-z}(\omega)\in G\cap A\},
\]
and
\[
B_\omega= \{z\in\Z^2\ :\ \tau_{-z}(\omega)\in G\cap B\}.
\]
$P(A)>0$ and $P(B)>0$, and therefore, by ergodicity of $P$, for $P$-almost every $\omega$, both $A_\omega$ and $B_\omega$ are non-empty. Note also that $A_\omega$ and $B_\omega$ are disjoint.

We now claim that for $P$-almost every $\omega$, every $z\in\Z^2$ and every neighbor $e$ of the origin, if $z\in A_\omega$ and $\omega(z,e)>0$ then $z+e\in A_\omega$. Due to shift invariance, it is sufficient to show this claim for $z=0$ and $e=e_1$. Indeed, let
$D=\{\omega\in A\cap G : \omega(0,e_1)>0 \mbox{ and } \tau_{-e_1}(\omega)\notin A\cap G\}$.
Then, for the generator $L$ of the process viewed from the point of view of the particle and the function $f={\bf 1}_{A\cap G}$, 
\[
0 = - \int MfdQ_A \geq \int_D\omega(0,e_1)dQ_A(\omega)
\]
where $M$ is the generator of the process w.r.t. the point of view of the particle as defined in \eqref{eq:genpart}.
The first equality follows from the stationarity of $Q_A$. This implies that $D$ is of measure zero, as desired.

The exact same thing holds for $B$ as well.

Therefore, for almost every $\omega$ there are even points in both $A_\omega$ and $B_\omega$, and thus for every coupling $\mu$ of a RWRE $X_n$ starting from such point in $A_\omega$ and $Y_n$ starting from such point in $B_\omega$, we get that
\[
\mu\left(
\exists_T \mbox{ s.t. } X_T=Y_T
\right)=0
\]
in contradiction to Lemma \ref{lem:coup}
\end{proof}

\begin{proof}[Proof of Lemma \ref{lem:invprin}]
As in the proof of the previous lemma, 
let $G=\{\omega\in\Omega\ :\ \frac{dQ}{dP}(\omega)>0\}$ and let
$G_\omega=\{z\in\Z^d\ :\ \tau_{-z}(\omega)\in G\}$. As in the proof of the previous lemma, $P$-almost surely, there exists an even point $y$ in $G$. By Fact \ref{fact:5.1}, for $P$-almost every $\omega$ and every $y\in G_\omega$, the law $\quenchedP_\omega^y$ converges to Brownian Motion. We need to prove the same thing for the law $\quenchedP_\omega^0$. To this end we find a coupling $\mu$ between a random walk $\{X_n\}$ which is distributed according to $\quenchedP_\omega^0$ and a random walk $\{Y_n\}$ which is distributed according to $\quenchedP_\omega^y$, such that 
\[
\mu\left(
\exists_T\mbox{ s.t. }\forall_{n>T}X_n=Y_n
\right)=1.
\]
$\mu$ exists by Lemma \ref{lem:coup}.

Now, $\mu$-almost surely there exists (a random) $C$ such that for all $N$,
\[
\frac{\max\big(\|X_n-Y_n\|\ :\ 1\leq n\leq N\big)}{\sqrt{N}}\leq\frac{C}{\sqrt{N}}
\mathop{\longrightarrow}_{N\to\infty}0
\]
Thus the invariance principle for $\{Y_n\}$ implies an invariance principle for $\{X_n\}$ as well.
\end{proof}
}

\section{Proof of Theorem \ref{thm:main}}\label{sec:high_d}
In this section we prove Theorem \ref{thm:main}. This follows from two statements: the first is that there exists a unique measure $Q$ which is invariant w.r.t. the point of view of the particle and is absolutely continuous w.r.t. $P$, and the second is that for every $z\in\Z^d$ the random walk starting from $z$ a.s. reaches the support (which we define below) of this measure $Q$ within finite time.

\subsection{The support of a stationary measure}\label{sec:defforergod}
For a measure $Q$ which is invariant w.r.t. the point of view of the particle and is absolutely continuous w.r.t. $P$, we define 
\[\supp Q=\{\omega:\frac{dQ}{dP}(\omega)>0\},\]
where the derivative is the Radon-Nykodim derivative. This is well define up to a set of $P$-measure zero.

For an $\omega\in\Omega$ and a measurable set $A\subseteq\Omega$ we define $A_\omega = \{z\in\Z^d:\tau_{-z}(\omega)\in A\}$.
%JDDNEWREV  For improvement of notation $\supp_\omega Q$ for $(\supp Q)_\omega$.
% applied. Noam
For improvement of notation we write $\supp_\omega Q$ for $(\supp Q)_\omega$.

\begin{claim}\label{lem:closedsupp}
For $P$-almost every $\omega$, every $z\in\Z^d$
%JDDNEWREV 2 to d
 and every neighbor $e$ of the origin, if $z\in \supp_\omega Q$ and $\omega(z,e)>0$ then $z+e\in \supp_\omega Q$.
\end{claim}

\begin{proof}
Due to shift invariance, it is sufficient to show this claim for $z=0$ and $e=e_1$. Let
$D=\{\omega\in \supp Q : \omega(0,e_1)>0 \mbox{ and } \tau_{-e_1}(\omega)\notin \supp Q\}$.
Then, for the generator $L$ of the process viewed from the point of view of the particle and the function $f={\bf 1}_{\supp Q}$, 
\[
0 = - \int LfdQ \geq \int_D\omega(0,e_1)dQ(\omega)
\]
%where $M$ is the generator of the process w.r.t. the point of view of the particle as defined in \eqref{eq:genpart}.
The first equality follows from the stationarity of $Q$. This implies that $D$ is of measure zero, as desired.
\end{proof}

\subsection{Ergodicity}\label{sec:ergod}
In this subsection we prove that there exists a unique measure $Q$ which is invariant w.r.t. the point of view of the particle and is absolutely continuous w.r.t. $P$.
\begin{lemma}\label{lem:finQ}
For every
probability measure $Q$ which is stationary w.r.t. the point of view of the particle and is absolutely continuous w.r.t. $P$,
\[
P(\supp Q)>\Phi,
\]
where $\Phi$ is as in Corollary \ref{cor:phi}.
\end{lemma}
\begin{proof}%[Proof of Lemma \ref{lem:finQ}]

By Claim \ref{lem:closedsupp}, $\supp_\omega(Q)$ is closed under the random walk (i.e. if $z\in\supp_\omega(Q)$ then $\quenchedP_\omega^{z}(\forall_n X_n\in \supp_\omega(Q))=1$) and therefore $\supp_\omega(Q)\cap \Delta_N$ is closed under the reflected random walk in $\Delta_N$. Therefore, for every $N$, there exists a stationary measure $Q_N$ which is supported on $\supp_\omega(Q)\cap \Delta_N$ and by Corollary \ref{cor:phi},
$P$-a.s. for all $N$ large enough $
|\supp Q_N| \geq \Phi|\Delta_N|.
$
Therefore
by the Ergodic Theorem $P(\supp(Q))\geq \Phi$.
 
\end{proof}

\begin{corollary}\label{cor:finmes}
There are finitely many probability measures that are stationary and ergodic w.r.t. the point of view of the particle and are absolutely continuous w.r.t. $P$. Further more, every $Q$ which is stationary w.r.t. the point of view of the particle and is absolutely continuous w.r.t. $P$, is a convex combination of these ergodic measures.
\end{corollary}

We now study the connectivity structure of $\supp_\omega Q$ for $Q$ ergodic. We start with a definition and then state and prove a few lemmas.
\begin{definition}\label{def:connected}
For $\omega\in\Omega$ and $x,y\in\Z^d$, we denote by  
$
x\connom y
$
the occurrence
\[
\quenchedP_\omega^x(\exists_n X_n=y)>0.
\]
We say that a set $A\subseteq\Z^d$ is {\em strongly connected} w.r.t. $\omega$ if for every $x$ and $y$ in $A$, 
$
x\connom y.
$
A set $A\subseteq\Z^d$ is called a {\em sink} w.r.t. $\omega$ if it is strongly connected and 
$x \hspace{-0.15cm} \not\connom y$ for every $x\in A$ and $y\notin A$.

\end{definition}

\begin{proposition}\label{prop:connected}
There exists $\kappa>0$ such that for every probability measure $Q$ which is stationary and ergodic w.r.t. the point of view of the particle and is absolutely continuous w.r.t. $P$,
 for $P$-a.e. $\omega$, $\supp_\omega Q$ contains a subset $A$ which is a sink w.r.t. $\omega$ and has upper density at least $\kappa$, i.e.  
\[
\limsup_{N\to\infty} \frac{\big|A\cap[-N,N]^d\big|}{\big|[-N,N]^d\big|}\geq\kappa.
\]
\end{proposition}

\begin{proof}
%First we claim that $P$-almost surely, for every $\epsilon$ and every $N$ large enough, there %exists a stationary and ergodic measure $H_N$ for the reflected walk on $\Delta_N$ which is %supported on $\Delta_N\cap\supp_\omega Q$ and
%\[
%\frac{|\supp H_N|}{|\Delta_N|}\geq \phi-\epsilon.
%\]
%
%Indeed, if we could find a sequence $N_k$ s.t. 
%\[
%\left\{\frac{|\supp H_{N_k}|}{|\Delta_{N_k}|}\right\}
%\]
%is bounded away from $\phi$ then a subsequential limit would give a stationary measure which %is supported on a strict subset of $\supp Q$, in contradiction to the ergodicity of $Q$. {\red %Further details in future versions}
%
%Since $\Delta_N$ is finite, $\supp H_N$ is strongly connected.
%
%Now, fix $K\in\N$ and $\delta>0$, and define the event
%\[
%B_{K,\delta}=\left\{
%0\in\supp_\omega Q
%\ ;\
%\left|\left\{
%x\in [-K,K]^d \cap\supp_\omega Q : 0 \hspace{-0.15cm} \not\connom x \mbox{ or } x \hspace{-0.15cm} \not\connom 0
%\right\}\right|>\delta K^d
%\right\}.
%\]
%The proposition will follow if we show that $P(B_{K,\delta})=0$ for every $K$ and $\delta$.
%
%To this end, remember that by the Ergodic Theorem, for $P$-almost every $\omega$
%\[
%P(B_{K,\delta})=\lim_{N\to\infty}\frac{1}{|\Delta_N|}\sum_{z\in\Delta_N}{\bf 1}_{\tau_{-z}(\omega)%\in B}.
%\]

For $P$-a.s. $\omega$ for all $N$ large enough the set $\Delta_N\cap\supp_\omega Q$ is non-empty and closed for the reflected random walk. Therefore there exists an ergodic measure $Q_N$ for the reflected random walk on $\Delta_N\cap\supp_\omega Q$. Note that $\supp Q_N$ satisfies three nice properties:
\begin{enumerate}
\item $|\supp Q_N|\geq\Phi|\Delta_N|$,
\item $\supp Q_N\subseteq\supp_\omega Q$, and
\item $\supp Q_N$ is a sink with respect to the reflected random walk under the environment $\omega$ on $\Delta_N$ (obviously, it cannot be a sink w.r.t. $\omega$ on the entire $\Z^d$).
\end{enumerate}

Fix $K\in\N$ and $\kappa>0$. We now define an event $B_{K,\kappa}$ as follows:
$B_{K,\kappa}$ is the event that the following things occur:
\begin{enumerate}
\item $0\in\supp_\omega Q$ (note that this is the same as $\omega\in \supp Q$).
\item There exists a set $A\subseteq [-K,K]^d\cap \supp_\omega Q$ such that
\begin{enumerate}
\item $|A|\geq\kappa|[-K,K]^d|$.
\item $0\in A$. In addition, $0\connom x$ and $x\connom 0$ for every $x\in A$.
\item $x \hspace{-0.15cm} \not\connom y$ for every $x\in A$ and $y\in[-K,K]^d\setminus A$.
\end{enumerate}
\end{enumerate}

%We make a claim.
\begin{claim}\label{claim:boccurs}
There exists $\alpha>0$ and $\kappa>0$ such that $P(B_{K,\kappa})>\alpha$ for all $K\geq 1$.
\end{claim}
We postpone the proof of Claim \ref{claim:boccurs}

Let $B_\kappa:=\{B_{K,\kappa}\mbox{ occurs for infinitely many values of $K$}\}$.
Using Claim \ref{claim:boccurs},
%JDDNEWREV kappa missing
 $P(B_\kappa)\geq\alpha$.

On the event $B_\kappa$, for every $K$ such that $B_{K,\kappa}$ occurs, let $A_K$ be the appropriate set. Then 
\[
\bigcup_{K:B_{K,\kappa}\mbox{ occurs}} A_K
\]
is a sink as required.

\end{proof}
\begin{proof}[Proof of Claim \ref{claim:boccurs}]
By Corollary \ref{cor:phi}, for every $N$ large enough there is a stationary measure $Q_N$ for the reflected random walk on $\Delta_N\cap\supp_\omega Q$ which is ergodic and such that $|\supp Q_N|>\Phi|\Delta_N|$. Fix some $\gamma$ and $\beta$ strictly between $0$ and $\Phi/2$. Take $N$ large which is divisible by $K$, and divide $\Delta_N$ into disjoint cubes $D_1,\ldots,D_{(N/K)^d}$.

For a cube $D_k$, we say that $D_k$ is good if at least $\beta|D_k|$ of the points in $D_k$ belong to $\supp Q_N$. We claim that at least proportion $\gamma$ of the cubes are good. Indeed, otherwise we get
$|\supp Q_N| \leq \gamma K^d(N^d/K^d) + \beta K^d(N^d/K^d) \leq (\gamma+\beta)N^d < \Phi|\Delta_N|$ which is a contradiction.

Now, by the ergodic theorem,
\[
P(B_{K,\kappa})
=\lim_{N\to\infty}\frac{1}{|\Delta_N|}\sum_{z\in\Delta_N}{\bf 1}_{B_{K,\kappa}}(\tau_{-z}(\omega))
\]
Now note that if we choose $\kappa=\beta\cdot 2^{-d}$, then $\tau_{-z}(\omega)\in B_{K,\kappa}$ for every $z$ which is in the intersection of $\supp H_N$ and a good cube. In this case, the set $A$ is simply the intersection of $H_N$ and $z+[-K,K]^d$. Now take $\alpha=\gamma\beta$. Then for all $N$ large enough which is divisible by $K$,

\[
\lim_{N\to\infty}\frac{1}{|\Delta_N|}\sum_{z\in\Delta_N}{\bf 1}_{B_{K,\kappa}}(\tau_{-z}(\omega))
\geq\alpha
\]
and therefore $P(B_{K,\kappa})\geq\alpha$.

\end{proof}

\begin{lemma}\label{lem:posdens}
\begin{enumerate}
\item\label{item:posdens}
For $P$-almost every $\omega$, every sink has lower density at least $\Phi/2^d$.
\item\label{item:finmany}
For every ergodic $Q$ which is invariant w.r.t. the point of view of the particle and is absolutely continuous w.r.t. $P$, $P$-a.s. there are only finitely many sinks contained in $\supp_\omega Q$.
\item\label{item:allpnt} $P$-a.s., every point in $\supp_\omega Q$ is contained in a sink.
\end{enumerate}
\end{lemma}
In other words, the lemma says that a.s. $\supp_\omega Q$ is a finite union of sinks, each of which has lower density at least $\Phi/2^d$.

\begin{proof}
Part \ref{item:posdens}: Let $S$ be a sink. Then for all $N$ large enough, $\Delta_N\cap S\neq\emptyset$. Therefore there is a stationary measure $H_N$ w.r.t. the reflected walk on $\Delta_N$ which is supported on $S$, and therefore, by Corollary \ref{cor:phi},
$|S\cap[-N,N]^d|\geq |S\cap\Delta_N|\geq |\supp H_N|\geq \Phi|\Delta_N|=(\Phi/2^d)|[-N,N]^d|$.
Part \ref{item:finmany} follows immediately from part \ref{item:posdens} and the fact that distinct sinks are disjoint. To see Part \ref{item:allpnt}, note that if $0$ is in a sink then any point reachable from $0$ is in a sink. Thus, if $A$ is the event that $0$ is in a sink, then $A$ is closed under the walk from the point of view of the particle. Therefore, $A\cap\supp Q$ is invariant under the walk from the point of view of the particle, and thus by ergodicity of $Q$, we get $Q(A)\in\{0,1\}$. Since we already proved $Q(A)>0$ we get $Q(A)=1$.
\end{proof}

\noindent
\begin{remark} In fact, we can also prove that $\supp_\omega Q$ is a sink (i.e. the finite number of sinks is one), but we do not do this now since we do not need it for our purposes.
\end{remark}

\begin{proposition}\label{prop:unique}
There exists a unique ergodic measure $Q$.
\end{proposition}

In what follows we use the following notation: For a set $A\subseteq\Z^d$, we denote its lower density by $\underline\dens (A)$, and its density, if such exists, by $\dens(A)$.

\begin{proof}
We use an adaptation of the easy part of the percolation argument of Burton and Keane \cite{BuKe}. Even though the finite energy condition is not satisfied, a very similar yet slightly weaker condition holds. In combination with the positive density of sinks (Lemma \ref{lem:posdens} Part \ref{item:posdens}) we can produce the percolation argument.
Let $Q_1$ and $Q_2$ be two distinct ergodic measures. Define $\dist(Q_1,Q_2):=\min(|z-w|:z\in\supp_\omega Q_1, w\in\supp_\omega Q_2)$. Note that due to shift invariance it is a $P$-almost sure constant, and therefore $\omega$ is rightfully omitted from the notation.
Let $z$ and $w$ be two points such that $|z-w|=\dist(Q_1,Q_2)$, and such that the event $U=U(z,w)=\{z\in\supp_\omega Q_1, w\in\supp_\omega Q_2\}$ has a positive $P$ probability. Let $i$ be a direction s.t. $\langle e_i,z-w\rangle\neq 0$.
 Let $R$ be the following measure on $\Omega\times\Omega$: we sample $\omega$ and $\omega'$. for all $x\neq z$, we take $\omega(x)=\omega'(x)$ to be sampled i.i.d. according to $\nu$. We then take $\omega(z)\sim (\nu|\omega(e_i)=0)$ and $\omega'(z)\sim (\nu|\omega(e_i)\neq  0)$. Again, everything is independent. Let $P_1$ be the distribution of $\omega$ and $P_2$ be the distribution of $\omega'$. Note that $P_1$ and $P_2$ are both absolutely continuous w.r.t. $P$, and that $P_1(U)>0$ and $P_2(U)=0$. Now let $\epsilon<\Phi/2^{d+5}$, and let $A\subseteq\Omega$ be an approximation of $\supp Q_1$, i.e. $P(A\bigtriangleup\supp Q_1)<\epsilon$ and $A\in\sigma(\omega(x):|x|<K)$ for some finite $K$. 
Now, for all $x$ s.t. $|x-z|>K$, we have that $x\in A_\omega$ if and only if $x\in A_{\omega'}$. Since both $P_1$ and $P_2$ are absolutely continuous w.r.t. $P$, we get that $R$ almost surely, by the ergodic theorem,
$\dens(A_\omega-\supp_\omega Q_1)=\dens(A_{\omega'}-\supp_{\omega'} Q_1)<\epsilon$ and equivalently $\dens(\supp_\omega Q_1-A_\omega)=\dens(\supp_{\omega'} Q_1-A_{\omega'})<\epsilon$. Therefore, a.s. conditioned on the event $\omega\in U$, we get $\dens(\supp_\omega Q_1-\supp_{\omega'} Q_1)<2\epsilon<\underline{\dens}(S_\omega)$ where $S_\omega$ is the sink containing $z$ in $\omega$. Therefore, $R$-a.s. on $\omega\in U$ there exist a point $x$ in $\supp_{\omega'}Q_1$ such that $x\connom z$. But then we also get $x\connomp z$, and thus $z\in\supp_{\omega'}Q_1$. Equivalently we get that $w\in\supp_{\omega'}Q_2$, and therefore $P_2(U)>0$ which is a contradiction. Therefore there exists a unique ergodic measure.
\end{proof}

\subsection{The probability of hitting $\supp_\omega Q$}\label{sec:hitQ}

In this subsection we show that with probability $1$ the walk has to hit $\supp_\omega Q$.

\begin{lemma}\label{lem:gtzero}
Let $Q$ be the probability measure which is stationary w.r.t. the point of view of the particle and is absolutely continuous w.r.t. $P$. Then for $P$-a.e. $\omega$ and every $z\in\Z^d$,
\[
\quenchedP_\omega^z\left(
\exists_N \mbox{ s.t. } \forall_{n>N} X_n\in\supp_\omega(Q)
\right)>0.
\]
%where
%\[
%\supp_\omega(Q) := \left\{
%z\in\Z^d : \frac{dQ}{dP}\big(\tau_{-z}\omega\big) > 0
%\right\}.
%\]
\end{lemma}

\begin{proof}
%By Lemma \ref{lem:finQ} and Corollary \ref{cor:finmes}, there are finitely many ergodic measures %w.r.t. the point of view of the particle which are absolutely continuous w.r.t. $P$. Call them %$Q_1,\ldots,Q_k$. Let
%\[
%Q = \frac 1k\sum_{i=1}^kQ_i.
%\]
%Then for every measure $Q'$ which is stationary w.r.t. the point of view of the particle and %absolutely continuous w.r.t. $P$, we have $Q' \ll Q$.
Assume for contradiction that there exists $B\subseteq\Omega$ such that $P(B)>0$ and for $\omega\in B$, there exists $z\in\Z^d$ such that
\[
\quenchedP_\omega^z\left(
\exists_N \mbox{ s.t. } X_N \in\supp_\omega(Q)
\right)=0.
\]
Then there exists $S\subseteq\Omega$ with $P(S)>0$ such that for every $\omega\in S$,
\[
\quenchedP_\omega^0\left(
\exists_N \mbox{ s.t. } X_N \in\supp_\omega(Q)
\right)=0.
\]

For every $\epsilon>0$ there exist $K\in\N$ and $A\subseteq\Omega$ such that $A$ is measurable w.r.t. $\sigma(\omega(z):\|z\|<K)$ and $P(A\bigtriangleup S)<\epsilon$.

$S$ is closed under the random walk and therefore $S \cap \Delta_N$ is closed under the reflected random walk in $\Delta_N$. Therefore, for every $N$, there exists a stationary measure $H_N$ which is supported on $S\cap \Delta_N$ and $P$-a.s. for all $N$ large enough satisfies $\left\|\frac{dH_N}{d{P_N}}\right\|_{\Delta_N,p}<C'$ for some $C'$.
% as in Lemma \ref{lem:finQ}.
Also, for $N$ large enough, $H_N(\tau_{-z}(\omega^{(N)}\notin A)< 2C'\epsilon$. As in Subsection \ref{sec:stnons}, let $H$ be a subsequential limit of $H_N$. Then $H$ is stationary w.r.t. the point of view of the particle. By Proposition \ref{prop:unique}, $H=Q$. In addition, $P(\supp(H))>C$ and $P(\supp(H)\setminus A)<2C'\epsilon$. However, 
$P(A\cap\supp(Q))\leq P(A\setminus S)\leq\epsilon$, which is clearly a contradiction.

\end{proof}

\begin{proposition}\label{prop:magia}
Let $Q$ be the probability measure which is stationary w.r.t. the point of view of the particle and is absolutely continuous w.r.t. $P$. Then for $P$-a.e. $\omega$ and every $z\in\Z^d$,
\[
\quenchedP_\omega^z\left(
\exists_N \mbox{ s.t. } \forall_{n>N} X_n\in\supp_\omega(Q)
\right)=1.
\]
\end{proposition}

\begin{proof}
Let
\[
h(z)=h_\omega(z)=
1-\quenchedP_\omega^z\left(
\exists_N \mbox{ s.t. } X_N \in\supp_\omega(Q)
\right).
\]

It suffices to show that $h\equiv 0$.
Assume for contradiction that with positive $P$-probability  there exists $z$ such that $h(z)>0$. Then by the ergodicity of $P$ w.r.t. the shifts, $P(\exists_z h(z)>0)=1$. We now show that $P$-almost surely, $\sup_zh(z)=1$.

Indeed, $h$ is a harmonic function w.r.t. the transition kernel, and therefore $h(X_n)$ is a martingale. Let $z$ be such that $h(z)>0$ and let $A$ be the (positive probability) event that the random walk starting at $z$ never hits $\supp_\omega Q$. By standard Martingale Theory, under the event $A$, the sequence $X_n$ has to converge to 1. \ignore{\red Write this later} 

Thus $\sup h=1$, but by Lemma \ref{lem:gtzero} the supremum is never attained. Now for $\eta>0$, let $h_\eta(z)=\eta+h(z)-1.$ Then, $P$-almost surely,
\begin{equation}\label{eq:conteta}
\sup h_\eta = \eta.
\end{equation}
However, for every large enough ball $B_r$ around the origin, by Theorem \ref{thm:mean_value} with power $p=1$, 
\[
\max_{B_r} h_\eta \leq C\cdot \max_{B_{2r}}h_\eta\cdot \frac{\#\{z\in B_{2r}:h_\eta(z)>0\}}{|B_{2r}|}.
\]
By taking a limit and using the ergodic theorem, we get
\[
\sup_{\Z^d}h_\eta \leq C\cdot\sup_{\Z^d}h_\eta\cdot P(h_\eta(0)>0).
\]

As $\lim_{\eta\to 0}P(h_\eta(0)>0)=0$, we get a contradiction 
%with \eqref{eq:conteta}
for all $\eta$ small enough. Therefore, $h\equiv 0$.

\end{proof}

\begin{proof}[Proof of Theorem \ref{thm:main}]
The theorem follows from Fact \ref{fact:5.1} and Propositions \ref{prop:unique} and \ref{prop:magia}.
\end{proof}

\ifarxiv
\else
\ignore{
\fi

\section{Proof of Theorem \ref{thm:mean_value}}\label{sec:mean_value_proof}
In this section we prove Theorem \ref{thm:mean_value}. The proof is a minor modification of the proof of Theorem 12 of \cite{GZ}. We do not include all details, but rather explain how to modify the proof in \cite{GZ} to our needs. The essential new ingredient is the use of Theorem \ref{thm:max_princ} and Lemma \ref{lem:for7} to control the stopping time $T_1$.
A difference between our notations and the notations in \cite{GZ} is that our generators $L_\omega$ and $L_\omega^{(N)}$ are the negatives of the corresponding generators in \cite{GZ}. We chose to do it this way for notational ease.

As in \cite{GZ}, we may choose $x_0=0$ and write $B_N=B_N(0)$.
Next take $k=k(N)=(\log N)^{100}$ as in the Remark \ref{rem:log100}, and set $T_1^{(N)}=\min(T_1,k)$ and
$$-L^{(N)}_\omega u(x)=\sum_{y}a^{(N)}_\omega(x,y)(u(y)-u(x))$$
where $a_\omega^N(x,y):=\quenchedP_\omega^x\big(T_1^{(N)}=y\big).$ Note that in view or Remark \ref{rem:balancedmart} this is balanced, i.e. 
\[
\sum_{y}(y-x)a_\omega^N(x,y) = 0.
\]

First note that if $L_\omega u(x)=0$ then using the optional stopping theorem,
$$L_\omega^{(N)}u(x)=0.$$
The next step is to adapt Lemma 14 of \cite{GZ}.
Fix $r>1$ and take $\beta = 2dr/ p> 2$ and set
$$\eta(x)=(1-|x|^2/N^2)^\beta 1_{|x|<N}$$  
and $v=\eta u^+$. Then the proof of Lemma 14 in \cite{GZ} shows that for every $x\in B_N$ with $I_v(x)\ne\emptyset$
\begin{equation}\label{eq:mf32}
L^{(N)}_\omega v(x)\le C(\beta)\eta^{1-2/\beta}N^{-2} h^2_x u^+(x)
\end{equation}
where
$$h_x=\max_{y: a^{(N)}_\omega(x,y)>0}|x-y|\leq (\log N)^{100}.$$
Our main problem is that $h_x$ is in our case unbounded in $N$.
We differentiate between points close to the boundary of $B_N$: 
$$B_{2,N}=\{x\in B_N: N^2-|x|^2< 4(\log N)^{100}\}$$
and points in the interior
$$B_{1,N}=\{x\in B_N: N^2-|x|^2\ge 4(\log N)^{100}\}$$
For $x\in B_{1,N}\cap \{x:I_v(x)\ne\emptyset\}$, following \cite{GZ} (27) and (28) we see that
$$\sum_{y}a^{(N)}_\omega(x,y)\frac{\eta(x)}{\eta(y)}(v(x)-v(y))\le \beta^2 2^{3\beta+1}N^{-2}\eta^{1-2/\beta}(x)\sum_{y}a^{(N)}_\omega(x,y)(x-y)^2 u^+(x)$$
and from (29) of \cite{GZ}
$$\sum_{y} a^{(N)}_\omega(x,y)\frac{\eta(x)}{\eta(y)}\big<s,y-x\big>\le \beta2^{2\beta+2}N^{-2}\eta^{1-2/\beta}(x)\sum_{y}a^{(N)}_\omega(x,y)(x-y)^2 u^+(x).$$
That is, in \eqref{eq:mf32} we can to replace $h^2_x$ with
$$
\sum_{y}a^{(N)}_\omega(x,y)(x-y)^2=E^x_\omega[T_1^{(N)}]=:T_\omega(x),
$$ 
where the equality follows from Wald's lemma,
and get
$$L^{(N)}_\omega v(x)\le \beta2^{4\beta+2}R^{-2}\eta^{1-2/\beta}(x)T_\omega(x)u^+(x).$$

%We first deal with points far from the boundary: 
%Take first $x$ such that $R^2-|x|^2\ge 4 R(\log R)^K$, and
% Then for $x\in B_R'(x)$ with $R'\le R-(\log R)^K$ such that 
%\[
%I_v(x)=\left\{
%\beta\in\R^d: \forall_{z\in \Delta_N\cup\partial{((\log R)^K)}\Delta_N}  v(z)\leq v(x)+\big<\beta,z-x\big>
%\right\}\ne\emptyset
%\]
%then following \cite{GZ} we get
%and
%$$\sum_{y} a^{(R)}_\omega(x,y)\frac{\eta(x)}{\eta(y)}\big<s,y-x\big>\le \beta2^{2\beta+2}R^{-2}\eta^{1-2/\beta}(x)\sum_{y}a^{(R)}_\omega(x,y)(x-y)^2 u(x).$$
%That is all we have to do is replace $h^2_x$ in \cite{GZ} 
%and therefore get for
%$\{x:R^2-|x|^2\ge 4 R(\log R)^K, I_v(x)\ne\emptyset\}$ 

Next, for $x\in B_{2,N}$ using the fact that the range of the walk is bounded by $(\log N)^{100}$ by \eqref{eq:mf32} we simply have 
$$L^{(N)}_\omega v(x)\le 2v(x)\le 32\eta^{1-2/\beta}N^{-2}(\log N)^{200}u^+(x).$$
%Write
%$$B_1=\{x\in B_R(x_0):R^2-|x|^2\ge 4 R(\log R)^K\}, \quad B_2=B_R(x_0)\setminus B_1,$$
In view of Remark \ref{rem:log100} we can  now apply the maximal inequality for $L^{(N)}_\omega$ so that for $N$ large enough, $P$ a.s.
\begin{eqnarray*}
\max_{B_N}v(x) &\le & 6 N^2\|1_{I_{D^\omega_v}\ne\emptyset}L^{N}_\omega v\|_{B_N,d}\\
&\le&
C\Big\|\eta^{1-2/\beta}u^+1_{B_{1,N}}T\Big\|_{B_N,d}+(\log N)^{200}C\Big\|\eta^{1-2/\beta}u^+1_{B_{2,N}}\Big\|_{B_N,d}.
\end{eqnarray*}

Next we use  H\"older's inequality: let $r'$ be such that $\frac1r+\frac1{r'}=1$. Then
$$\Big\|\eta^{1-2/\beta}u^+1_{B_{1,N}}T\Big\|_{B_N,d}\le \|T\|_{B_N,dr'}\Big\|\eta^{1-2/\beta}u^+\Big\|_{B_N,dr}$$
and
$$(\log N)^{200}\Big\|\eta^{1-2/\beta}u^+1_{B_{2,N}}\Big\|_{B_N,d}\le (\log N)^{200}\|1_{B_{2,N}}\|_{B_N,dr'}\Big\|\eta^{1-2/\beta}u^+\Big\|_{B_N,dr}.$$
Take $N\ge N_1=N_1(r)$ such that
$$(\log N_1)^{200}\|1_{B_{2,N_1}}\|_{B_{N_1,dr'}}\le C (\log N_1)^{200}N_1^{-1/(2dr')}\le 1,$$
and using the ergodic theorem and Lemma \ref{lem:for7} we now that $P$ a.s. for $N\ge N_2=N_2(dr')$,
$$\|T_\omega\|_{B_{N,dr'}}\le C=C(dr')<\infty.$$
Thus for each $r>1$, if   $N\ge N_0(r)=\max(N_1,N_2)$  there exist $C=C(r)$ such that 
$$\max_{B_N}v(x)\le C \Big\|\eta^{1-2/\beta}u^+\Big\|_{B_{N,dr}},$$
and can then proceed as in the proof of Theorem 12 of \cite{GZ}.

\hfill\qed

\ifarxiv
\else
}
\fi

\section{Concluding remarks}\label{sec:ques}

We end this paper with a number of remarks and open questions.

%\begin{iquestion}
%\end{iquestion}

\begin{iremark}
Our result is also true for time continuous balanced RWRE
generated by $L_\omega$. One way of seeing it is that by the Ergodic theorem the time scales of both processes are comparable.
\end{iremark}

\begin{iremark}
Although not done here, we believe that our result extends easily to
i.i.d. genuinely d-dimensional (appropriately defined) finite range balanced environments, that is for which
$$\sum_{z\in\Z^d} \omega(x,z)(z)=0$$
with
$$\omega(x,z)=0,\qquad\text{if}\quad |z|\ge R$$
for some $R\ge 1$,
since the essential analytical tools work for such generators.
Note that this is less restrictive than strongly balanced
$$\omega(x,z)=\omega(x,-z), \quad \forall z.$$
Of course both definitions agree in the nearest neighbor case.
\end{iremark}

\begin{iremark}
A much more challenging problem is to add a deterministic drift.
For example take for $\epsilon\in (0,1)$
$$\omega(x,e)=(1-\epsilon)\omega_0(x,e)+\epsilon 1_{e=e_i}$$
where $\omega_0$ is i.i.d. balanced, genuinely d-dimensional.
%Then the centered walk $\bar X_n=X_n-n\epsilon e_i$ should
%satisfy the quenched invariant principle as well.
\end{iremark}

\begin{iremark}
Replacing the i.i.d. hypothesis with a strongly mixing condition
on the environment is also a natural question.
Example \ref{exam:noclt} shows that general ergodic (and even mixing) media do not satisfy the quenched
invariance principle, but things could be manageable if the environment has strong enough 
mixing conditions.
\ignore{
In the 2-d case, we can achieve good control both on the percolation problem and on the moments of the rescaling time, and thus the methods of this paper will work with very little change. In higher dimensions, both estimates rely strongly on the i.i.d. structure, and thus new ideas are necessary.

This is similar to the situation in the random conductance model,
where the 2-d cases are generally more robust (see \cite{BB}), while in higher dimensions situations which are not uniformly elliptic and not i.i.d. are much harder to handle.
}

\end{iremark}

\begin{iremark}
 The percolation problem in higher dimensions on its own is a source
of open questions. One interesting questions is:  are all infinite strongly connected components sinks
or are there also other components? This is essentially the question of uniqueness of 
the infinite strongly connected component.
\end{iremark}

And finally,
\begin{iremark}
Can we get heat kernel bounds of the Aronson type at large scale
or Harnack inequalities? See e.g. \cite{barlow} where this is done in a non-elliptic reversible setting.
\end{iremark}

\section*{Acknowledgment}
Xiaoqin Guo and Ofer Zeitouni are acknowledged for numerous helpful  discussions. Max von Renesse is acknowledged for many useful discussions.

\bibliography{balanced}
\bibliographystyle{plain}

\end{document}